%% file: IETI_RicamReport.tex
\documentclass[runningheads,12pt]{llncs}
\usepackage [english]{babel}
\usepackage[utf8]{inputenc}
\usepackage[T1]{fontenc}

\usepackage{amssymb}
\usepackage{amsmath}
\usepackage{amsfonts}

\usepackage{graphicx}
\usepackage{subfigmat}
\usepackage{url}

\usepackage{latexsym} 
\usepackage{ wasysym}
\usepackage{newlfont}
\usepackage{leqno}
\usepackage[usenames]{color}
\usepackage{fancybox}
\usepackage{sidecap}
\usepackage{wrapfig}
\usepackage{sidecap}
\usepackage{algorithm}

  \usepackage{times}
  \usepackage{parskip}
  \usepackage{algpseudocode}

  \newtheorem{assumption}{Assumption}

\usepackage{stmaryrd}

{\begin{Sbox}\begin{minipage}}%
{\end{minipage}\end{Sbox}\fbox{\TheSbox}}

    \urldef{\mailsa}\path|ulrich.langer@ricam.oeaw.ac.at|
    \urldef{\mailse}\path|christoph.hofer@ricam.oeaw.ac.at|
    \newcommand{\keywords}[1]{\par\addvspace\baselineskip  
     \noindent\keywordname\enspace\ignorespaces#1}

\newcommand{\PathToPic}{./Pictures/}

\input{Makros.tex}

\begin{document}
\mainmatter  


\title{Dual-Primal Isogeometric Tearing and Interconnecting Solvers
    for 
    large-scale 
    systems of multipatch 
    continuous Galerkin
    IgA equations}

\titlerunning{IETI-DP for solving multipatch IgA equations}


\author{
Christoph Hofer$^1$
\and     Ulrich Langer$^1$ 
}
\authorrunning{C. Hofer,  U. Langer}

\institute{ $^1$ Johann Radon Institute for Computational and Applied Mathematics (RICAM),\\
Austrian Academy of Sciences\\ 
\mailse\\
\mailsa \\
 }

%

\maketitle

\begin{abstract}
{
   The dual-primal isogeometric tearing and interconnecting (IETI-DP) method is the adaption of the dual-primal finite element tearing and interconnecting (FETI-DP) method to isogeometric analysis of scalar elliptic boundary value problems like, e.g., diffusion problems with heterogeneous diffusion coefficients. The purpose of this paper is to extent the already existing results on condition number estimates to multi-patch domains, which consist of different geometrical mappings for each patch. Another purpose is to prove a 
polylogarithmic condition number bound for the preconditioned system with stiffness scaling in case of $C^0$ smoothness across patch interfaces. Numerical experiments validate  the presented theory.
}
\end{abstract}

\section{Introduction}
Isogeometric Analysis (IgA) is a new methodology for solving partial differential equations (PDEs) numerically. 
IgA was introduced by  Hughes, Cottrell and  Bazilevs 
in 
\cite{HL:HughesCottrellBazilevs:2005a},
and has become a very active field of research,
see also \cite{veiga2006IGA_aopproximation}
for the first results on the numerical analysis of IgA,
and the monograph  \cite{Hughes2009_IGAbook} for a comprehensive presentation of the IgA.
%
The main idea is to use the basis functions, which are used for the representation 
of the geometry in computer aided design (CAD) models anyway, also for the approximation
of the solution of the  PDE or the system of PDEs describing the physical 
phenomenon which we are going to simulate.
The typical choice for such basis functions are B-Splines or Non-Uniform Rational Basis Spline (NURBS). 
One advantage of the IgA over the more traditional  finite element method (FEM) is certainly the fact 
that there is no need for decomposing the computational domain into finite elements.
Hence, one gets rid of this geometrical error source, at least, in the class of computational domains
that are produced by a CAD system.
Moreover, it is much easier to build up $C^l, l\geq 1$, conforming basis functions in IgA than in the 
finite element (FE) case. The major drawback is the fact that the basis functions are not nodal and have a larger support. 
However, it is still possible to associate basis functions to the interior, the boundary and the vertices of the domain, which is crucial for the dual-primal isogeometric tearing and 
interconnecting (IETI-DP) method, which was introduced in \cite{HL:KleissPechsteinJuettlerTomar:2012a}. The IETI-DP method is an extension of the dual-primal finite element tearing and interconnecting method (FETI-DP)  to IgA. 
A comprehensive 
theoretical analysis of the FETI-DP method can be found in, e.g., \cite{pechstein2012FETI}, \cite{toselli2005DDM} or \cite{Tarek_2008_DDM_FEM}. 
We here also mention the recent development of other IgA domain decomposition (DD) techniques. 
In particular, we want to mention isogeometric overlapping Schwarz methods, where we refer to \cite{beirao2012_OverlappingSchw}, \cite{beirao2013_Overlapping_LinElasicity}, \cite{bercovier2015overlapping}, and isogeometric mortaring discretizations, see \cite{hesch_2012MortaringDiscr}.
The analysis of Balancing Domain Decomposition by Constraints (BDDC) preconditioner, 
which has been done for IgA matrices in \cite{beirao2013IGA_bddc},
also applies to the IETI-DP method 
due to the same spectrum (with the exception of at most two eigenvalues), 
see \cite{ManDohrTez2005_AlgebTheoryFETIDP_BDDC}. 
Based on the FE work in \cite{Dohrmann2012_FEM_DeluxeScal}, 
a recent improvement for the IgA BDDC preconditioner with a more advanced scaling technique, the so called \emph{deluxe scaling}, can be found in \cite{veiga2014IGA_BDDC_deluxe}.  

The goal of this paper is to extend the condition number estimates for BDDC preconditioners,
 presented in \cite{beirao2013IGA_bddc}, to multipatch domains 
composed of non-overlapping patches which are images of the parameter domain 
by several different geometrical mappings.
Moreover, we present the derivation of an improved 
bound for the condition number 
for the so-called
\emph{stiffness scaling} in the simplified case of $C^0$ smoothness across patch interfaces. 
It turns out that, in our 
case, the stiffness scaling provides the same quasi optimal 
polylogarithmic bound in terms of $H/h$ as for the coefficient scaling.

In the present paper, we consider the following model elliptic boundary problem 
in a bounded Lipschitz domain $\Omega\subset \real[d], d\in\{2,3\}$:
find $u: \overline{\Omega} \rightarrow \mathbb{R}$ such that\\
%
%
\begin{equation}
  \label{equ:ModelStrong}
  - \mdiv(\alpha \grad u) = f \; \text{in } \Omega,\;
  u = 0 \; \text{on } \Gamma_D,
  \;\text{and}\;
  \alpha \DN{u} = g_N \; \text{on } \Gamma_N,
\end{equation}
with given, sufficient smooth data $f, g_N \text{ and } \alpha$, where  
the coefficient $\alpha$ is uniformly bounded from below and above
by some positive constants $\alpha_{min}$ and $\alpha_{max}$, respectively.
The boundary $\Gamma = \partial \Omega$ of the computational domain $\Omega$
consists of a Dirichlet part $\Gamma_D$ and a Neumann part $\Gamma_N$.
Furthermore, we assume that the Dirichlet boundary $\Gamma_D$ is always 
a union of complete domain sides which are uniquely defined in IgA.. 
Without loss of generality, we assume 
homogeneous
Dirichlet conditions. 
This can always be  obtained by homogenization.

By means of integration by parts, 
we arrive at
the weak formulation of \equref{equ:ModelStrong} 
which reads as follows:
find $u \in \VD = \{ u\in H^1: \gamma_0 u = 0 \text{ on } \Gamma_D \}$
such that
\begin{align}
  \label{equ:ModelVar}
    \bil{u}{v} = \func{F}{v} \quad \forall v \in \VD,
\end{align}
where $\gamma_0$ denotes the trace operator. The bilinear form
$a(\cdot,\cdot): \VD \times \VD \rightarrow \mathbb{R}$
and the linear form $\func{F}{\cdot}: \VD \rightarrow \mathbb{R}$
are given by the expressions
\begin{equation*}
\bil{u}{v} = \int_\Omega \alpha \nabla u \nabla v \,dx
\quad \mbox{and} \quad
\func{F}{v} = \int_\Omega f v \,dx + \int_{\Gamma_N} g_N v \,ds,
\end{equation*}
respectively.
%
%
  
The rest of the paper is organized as follows. In \secref{sec:IGA}, we recall the basic definitions and properties of B-Splines as well as the main principles of IgA. The IETI-DP and the corresponding BDDC methods are explained in \secref{sec:IETI}. In particular, 
the new results on the condition number estimate, as mentioned above,
are presented in Subsection~\ref{sec:AnalysisBDDC}.
Section~\ref{sec:implementation} is devoted to the implementation of the IETI-DP method.
The numerical examples confirming the theory are presented in Section~\ref{sec:numExamplesIETI}.
Finally, in Section~\ref{sec:conclusions}, we draw some conclusions and discuss
further issues concerning generalizations to multipatch discontinuous Galerkin IgA schemes 
as constructed and analysed in 
\cite{HL:LangerMoore:2014a} and \cite{HL:LangerToulopoulos:2014a}.


\section{Some preliminaries on Isogeometric analysis }
\label{sec:IGA}
B-Splines and NURBS
play an important role in computer aided design and computer graphics. Here we will use these splines for building our 
trial and test spaces 
for the Galerkin approximations to \equref{equ:ModelVar}, as proposed in \cite{Hughes2009_IGAbook}. This section provides the definition of B-Splines in one dimension and in higher dimensions via a tensor product structure. We will give an overview of isogeometric discretization and summarize the approximation properties of these B-Splines and NURBS.


\subsection{Univariate B-Splines}
\label{sec:IGA:UnivariateB-Splines}
\begin{definition}
Let $\kn[1] \leq \ldots \leq \kn[\nkn]$ be a finite, real-valued, monotonically increasing sequence of real numbers. The set
 $\{\kn[1],\ldots,\kn[\nkn]\}$ is called \emph{knot vector}. An entry $\kn[i], i\in\{1,\ldots,\nkn\}$ is called \emph{knot} and is called an 
 \emph{interior knot} if $(\kn[1] < \kn[i]) \wedge (\kn[i] < \kn[\nkn])$. If $r$ knots have the same value, we say that the knot has multiplicity $\mul\def\@thmcountersep{.}$, i.e.,
 $\mul = |\{j\in\{1,\ldots,\nkn\}: \kn[j] = \kn[i]\}|$ 
 is the cardinal number of the set $\{j\in\{1,\ldots,\nkn\}: \kn[j] = \kn[i]\}$. The interval between two knots is called \emph{knot span}. A knot span is called \emph{empty} if
 $\kn[i] = \kn[i+1]$ and is called interior if $\kn[1]<\kn [i+1] \wedge \kn[i]<\kn[\nkn]$. The knots provides a partitioning of the parameter domain into elements. 
 If the knots are equally spaced in the domain, we call it \emph{uniform}, otherwise \emph{non-uniform}. 
\end{definition}
Based on a knot vector, we can define the B-Spline functions recursively.
\begin{definition}
 Let $\deg\in\nat$ and $\knV$ be a knot vector with multiplicity of any interior knot of at most $\deg$. Then by the \emph{Cox-de Boor} recursion formula we 
 can define the $\nBasis = \nkn - \deg -1$ \emph{univariate B-Spline} basis functions on $[\kn[1],\kn[\nkn]]$ as follows:
 \begin{align}
 \label{equ:BSpline0}
  \BSplG{i}{0}(\kn) &=\begin{cases}
                 1 & \mbox{if }\kn[i]\leq \kn \leq \kn [i+1]\\
                 0 & \mbox{otherwise}
                \end{cases},\\
  \label{equ:BSpline}
  \BSpl(\kn)&= \frac{\kn-\kn[i]}{\kn[i+\deg]-\kn[i]}\BSplG{i}{\deg-1}(\kn) + \frac{\kn[i+\deg+1]-\kn}{\kn[i+\deg+1]-\kn [i+1]}\BSplG{i+1}{\deg-1}(\kn),
 \end{align}
 where $i=1,\ldots,\nBasis$.
 If in \equref{equ:BSpline} appears a $0/0$ we define it as $0$. The number $\deg$ is then called \emph{degree} of the B-Spline.
\end{definition}
\begin{definition}
 Let $\knV$ be a knot vector. We say that the knot vector is \emph{open} if the multiplicity of the first knot and the last knot are $\deg+1$ whereas the multiplicity of the 
other knots is  $\deg$ at most. 
\end{definition}

B-Splines defined on open knot vectors are interpolatory at the beginning and the end of the parameter interval, while all other basis function are zero there. Hence it is possible 
to identify basis functions corresponding to the interior and the boundary. Additionally, the number of interior knot spans $\nIntKnotspan$ is given by
$\nIntKnotspan = \nkn - 1 - 2\deg.$
Without loss of generality,  we can restrict ourselves to a certain class of knot vectors
by means of a suitable scaling.
\begin{assumption}
We only consider knot vectors being a partition of $[0,1]$, i.e.
$\kn[1] = 0$ and  $\kn[\nkn]= 1$.
\end{assumption}
Since later on, we will only be interested in $C^0$ continuity across the interfaces, we restrict our analysis to open knot vectors. 
\begin{assumption}
 We consider all  knot vectors used as open knot vectors.
\end{assumption}
%

At the end of this section, let us summarize some important properties of B-Splines:
\begin{enumerate}
\item The B-Splines basis functions $\BSpl$ form a partition of unity, i.e.
      \begin{align*}
      \sum_{i=1}^{\nBasis} \BSpl(\kn) \equiv 1 
      \end{align*}
      for all $p=0,1,\ldots$ .
\item The B-Spline basis functions are non negative, i.e. $\forall \kn\in[\kn[1],\kn[\nkn]]\, \forall i\in\{1,\ldots,\nBasis\}: \BSpl(\kn)\geq 0$.
\item The support $\BSpl$ is local and it holds
    \begin{align*}
    \supp \BSpl \subseteq (\kn[i],\kn[i+\deg+1]), \quad \forall i\in\{1,\ldots,\nBasis\}.
    \end{align*}
\item Only $\deg+1$ basis functions are non zero on a knotspan $(\kn[i],\kn[i+1])$ and its 		indices are $i-\deg,\ldots,i$, i.e.
  \begin{align*}
   \BSplG{j}{\deg}|_{(\kn[i],\kn[i+1])} \neq 0 \Leftrightarrow j\in \{i-\deg,\ldots,i\},
  \end{align*}
  where $i\in\{1,\ldots,\nkn-1\}$.
\item On each knotspan the B-Spline basis functions are piecewise polynomials of degree 	$\deg$ and, without multiple knots in the interior, $C^{\deg-1}$ continuous.
  At a knot with multiplicity $r$, it has $C^{\deg-r}$ continuity. 
  Hence, the continuity is reduced in the presence of multiple knots.
  \item If a knot $\kn[l]$ has the multiplicity $r = \deg$, then there is one basis function $\BSpl$, such $\BSpl(\kn[l])=1$ and all other basis functions have zero value there,
      i.e.,  the basis is interpolatory at $\kn[l]$.
\end{enumerate}


\subsection{Tensor product B-Splines}

In order to define B-Splines in higher dimensions, we make use of the tensor product.
\begin{definition}
 Let $(\deg^1,\ldots,\deg^\pDim)$ be a vector in $\nat^\pDim$, 
and let, 
for all $\iVar=1,\ldots,\pDim$, 
$\knV^\iVar$ be a knot vector. 
Furthermore, we denote 
the $i^\iVar$ univariate B-Spline defined on the knot vector $\knV^\iVar$
by $\BSplG{i^\iVar}{\deg}^\iVar(\kn^\iVar)$. 
Then the $\pDim$-dimensional tensor product B-Spline (TB-Spline) is defined by
 \begin{align}
 \label{def:TBSpline}
  \BSplG{(i^1,\ldots,i^\pDim)}{(\deg^1,\ldots,\deg^\pDim)}(\kn) = \prod_{\iVar=1}^{\pDim}\BSplG{i^\iVar}{\deg^\iVar}^\iVar(\kn^\iVar).
 \end{align}
\end{definition}
In order to avoid cumbersome notations, we will again denote the tensor product B-Spline by $\BSpl$ and interpret $i$ and $\deg$ as multi-indices. Additionally we define the set 
of multi-indices $\iSet$ by
\begin{align*}
 \iSet := \{(i^1,\ldots,i^\pDim): i^\iVar \in \{1,\ldots,\nBasis_\iVar\} \, \forall \iVar\in\{1,\ldots,\pDim\}\},
\end{align*}
where $\nBasis_\iVar$ are the number of B-Spline basis function for dimension $\iVar$.
Due to the tensor product structure, the TB-Splines provide the same properties as the univariate B-Splines.

In case of TB-Spline, we will call a non-empty knotspan $\pCell[i] = (\kn[i],\kn[i+1]), i\in\iSetMesh$ also \emph{cell}, where $\iSetMesh = \{ i\in\iSet| i^\iVar\neq \nBasis_\iVar\}$ and $(\kn[i],\kn[i+1])$ is defined as
\begin{align*}
 (\kn[i],\kn[i+1]) := (\kn[i^1]^1,\kn[i^1+1]^1)\times\ldots\times(\kn[i^\pDim]^\pDim,\kn[i^\pDim+1]^\pDim).
\end{align*}
The mesh created by these cells is denoted by $\pMesh$, i.e.
\begin{align*}
 \pMesh := \{ (\kn[i],\kn[i+1]) | i \in \iSetMesh \}.
\end{align*}


\subsection{B-Spline geometries and geometrical mapping}

The B-Splines are used to represent a $\pDim$-dimensional geometry in $\real[\gDim]$, 
where $\pDim\leq \gDim$. 
In the following, we will restrict ourselves to the case $\pDim = \gDim\in \{2,3\}$. 
\begin{definition}
 Let $\BSplFam$ be a family of tensor B-Spline basis functions. Given \emph{control points} $\Cp[i]\in\real[\gDim]$, $i\in\iSet$,
the \emph{B-Spline surface/volume} is defined by
\begin{align*}
  \geoMap :\;& \pDom \rightarrow \real[\gDim]\\
	    &\geoMap(\kn) := \sum_{i\in \iSet}\Cp[i]\BSpl(\kn).
\end{align*}
 We call $\geoMap$ the \emph{geometrical mapping}, the domain $\pDom$ of $\geoMap$ \emph{parameter domain} and the image $\gDom := \geoMap(\pDom)\subset \real[\gDim]$ \emph{physical domain}.
 The geometrical mapping is called \emph{regular} if $\det\nabla \geoMap(\kn)\neq0,\,\forall \kn\in[0,1]^\pDim$.
\end{definition}
The knot vector $\knV$ provides a partition of the parameter domain 
into cells, and, by means of the geometrical mapping, 
we receive a partition of the physical space in cells $\gCell[i]$ as well, where
\begin{align*}
  \gCell[i]= \geoMap(\pCell[i]), \pCell[i] \in \pMesh.
\end{align*}
If we collect all these cells, then we get a mesh $\gMesh$ for the physical domain 
\begin{align*}
   \gMesh := \{\gCell = \geoMap(\pCell)| \pCell \in \pMesh\}.
\end{align*}
%
%
\begin{definition}
A family of meshes $\{\gMesh\}_{h\in H}$ is called \emph{quasi uniform}, if there exists a constant $\theta\geq 1$ for all $\gMesh \in \{\gMesh\}_{h\in H}$, such that $\theta^{-1} \leq \diam(\gCell)/\diam(\gCell')\leq \theta$ for all $ \gCell, \gCell'\in \gMesh $.
\end{definition}

\begin{assumption}
All considered meshes are quasi uniform.
\end{assumption}


\subsubsection{Multi-patch geometries}
%
In many practical applications, it is not possible to describe the physical computational 
domain $\gDom$ just with one geometrical mapping $\geoMap$.
Therefore, we represent the physical domain $\gDom$ by $\nMP$ non overlapping domains $\gDom\sMP$, called \emph{patches}. Each $\gDom\sMP$ is the image of an associated geometrical mapping $\geoMap^{(k)}$, defined on the parameter domain $\pDom$, i.e.,
$\gDom\sMP = \geoMap\sMP\left(\pDom\right) \; \for \iMP = 1,\ldots,\nMP$,
and
$\overline{\gDom} = \bigcup_{\iMP=1}^{\nMP} \overline{\gDom}\sMP.$
%
Clearly, each patch has a mesh $\gMesh\sMP$ in the physical domain and a mesh $\pMesh\sMP$ in the parameter domain, consisting of cells $\gCell\sMP$ and $\pCell\sMP$.

We denote 
the interface between the two patches $\gDom\sMP[k]$ and $\gDom\sMP[l]$ by $\gInt$,
and 
the collection of all interfaces by $\gIntMP$, i.e.,
\begin{equation*}
\gInt = \overline{\gDom}\sMP[k]\cap\overline{\gDom}\sMP[l]
\quad \mbox{and} \quad
\gIntMP := \bigcup_{l>k}\gInt.
\end{equation*}
%
%
Furthermore, the boundary of the domain is denoted by $\partial\Omega$. Note that the interface $\gIntMP$ is sometimes called \emph{skeleton}.

\subsection{Isogeometric discretization}
\label{sec:IGA-disc}

The key point in isogeometric analysis is the use of the same functions for representing the geometry as well as basis functions for the solution space. This motives the definition of the basis functions in the physical domain via the push-forward of the basis functions in the parameter domain, i.e.,
\begin{align*}
  \gBSpl:= \BSpl \circ \geoMap^{-1}.
\end{align*}
Thus, we 
define 
our 
discrete function space 
$\gVh$ by
 \begin{align}
   \gVh = \spn\gBSplFam \subset H^1(\gDom).
 \end{align}
The function $u_h$ from the IgA space $\gVh$ can therefore be represented 
in the form
 \begin{align*}
  u_h(\x) = \sum_{i\in\iSet} u_i \gBSpl(\x).
 \end{align*}
 Hence, each function $u_h(\x)$ is associated with the vector $\ub = (u_i)_{i\in\iSet}$. 
This map is known as \emph{Ritz isomorphism}. One usually writes this relation as 
$u_h \leftrightarrow \ub$,
%
%
 and we will use it in the following without further comments.

 Sometimes we also need the space $\pVh$ of spline functions in the parameter domain,
i.e.,
 \begin{align}
  \pVh = \spn\BSplFam \subset H^1(\pDom).
 \end{align}

Similar to the previous subsection, we define the discrete function spaces for each patch of a multipatch by
 \begin{align}
 \label{def:gVhMP}
   \gVh\sMP = \spn\gBSplFamG[k]{i}{\deg} \subset H^1(\gDom\sMP)
 \end{align}
 and functions therein by
 \begin{align*}
  u_h(\x) = \sum_{i\in\iSet\sMP} c_i\sMP \gBSpl\sMP(\x),
 \end{align*}
 where
\begin{align*}
  \gBSpl\sMP:= \BSpl\sMP \circ {\geoMap\sMP}^{-1}
\end{align*}
The discrete function spaces for the whole multipatch domain is then given by
 \begin{align*}
   \gVh = \{v| \;v_{|\gDom\sMP}\in\gVh\sMP\}\cap H^1(\gDom).
 \end{align*}
The space of spline functions in the parameter domain is analogously defined as $\pVh\sMP$.

Based on the work in \cite{beirao2013IGA_bddc}, we can find an important splitting of the space $\gVh$. Since we are using open knot vectors, we can identify basis function on the interface $\gIntMP$ and in the interior of each patch. The set of all indices of basis function having a support on the interface is denoted by $\setB$ and the corresponding space is defined via
\begin{align}
\label{def:VB}
 \gVB := \spn\{ \gBSpl|\; i\in \setB\}\,\subset H^1(\gDom).
\end{align}
For a function $u\in\gVB$, we define its restriction to a single patch $\gDom\sMP$ by $u\sMP\in\gVB\sMP$, where $\gVB\sMP:=\{u_{|\gDom\sMP}|u\in\gVB\}$. Similarly, we define the space of all functions living in the interior of $\gDom\sMP$ by
\begin{align}
\label{equ:VIMP}
 \gVI\sMP := \gVh\sMP \cap H^1_0(\gDom\sMP).
\end{align}
It is easy to see that the space $\gVI\sMP$ has also the following representation
\begin{align}
\label{equ:VIMP2}
 \gVI\sMP = \spn\{ \gBSpl\sMP|\; \supp\{\gBSpl\sMP\}\subset\gDom\sMP\}.
\end{align}
This leads to the decomposition
\begin{align*}
 \gVh = \prod_{\iMP =1}^{\nMP}\gVI\sMP \oplus \DHE[\gVB],
\end{align*}
where $\DHE$ is the \emph{discrete NURBS harmonic extension} defined by
\begin{align}
\label{def:DHE}
\begin{split}
 \DHE&: \gVB \to \gVh:\\
&\begin{cases}
 \text{Find }\DHE{v\sB}\in \gVh: & \\
  \quad\bil{\DHE{v\sB}}{v\sMP}=0 \quad &\forall v\sMP\in \gVI\sMP, \, 1\leq\iMP\leq\nMP,\\
   \quad \DHE{v\sB}_{|\partial\gDom\sMP} = {v\sB}_{|\partial\gDom\sMP} & 1 \leq \iMP\leq\nMP.
\end{cases}
\end{split}
\end{align}
%
See \cite{beirao2013IGA_bddc} and \cite{Schumaker2007_Spline_functions} for a more sophisticated discussion.

\subsubsection{Continuous Galerkin IgA Schemes}
\label{sec:DiscreteProblem}
We now look for the Galerkin approximate $u_h$ from the finite dimensional subspace 
$\VDh$ of $\VD$, where $\VDh$ is the set of all functions from $\gVh$ 
which vanish on the Dirichlet boundary $\Gamma_D$.
The Galerkin IgA scheme reads as follows:
find $u_h \in \VDh$ such that
  \begin{align}
  \label{equ:ModelDisc}
    \bil{u_h}{v_h} = \func{F}{v_h} \quad \forall v_h \in \VDh.
  \end{align}
  
A basis of the space $\VDh$ are the B-Spline functions $\{\gBSpl\}_{i\in\iSet_0}$. However, we have to exclude those basis functions which have a contribution to the value on the Dirichlet boundary, {obtaining the new index set $\iSet_0$}. Since we are using open knot vectors, we can identify those basis functions. By choosing such a basis and introducing a global ordering $\iSet$ of the basis function across all patches, 
we can rewrite the Galerkin IgA scheme \equref{equ:ModelDisc} 
as a linear system of IgA equations of the form
\begin {align}
\label{equ:Ku=f}
  \Kb \ub = \fvec,
\end{align}
where 
$\Kb = (\Kb_{i,j})_{i,j\in {\mathcal{I}}_0}$
and
$\fvec= (\fvec_i)_{i\in {\mathcal{I}}_0}$
denote the stiffness matrix and the load vector, respectively,
with 
$ \Kb_{i,j} = \bil{\gBSplG{j}{\deg}}{\gBSplG{i}{\deg}}$
and 
$\fvec_i = \func{F}{\gBSplG{i}{\deg}}$,
%
and $\ub$ is the vector representation of $u_h$ 
given by the IgA isomorphism. {In order to keep the notation simple, we will reuse   the symbol $\iSet$ for the set $\iSet_0$} in the following.
 
\subsubsection{Schur complement system}
\label{sec:Schurcomp}
Introducing the bilinear form 
\begin{align*}
 s:& \gVB \times \gVB \to \real\\
  &s(w\sB,v\sB) = \bil{\DHE{w\sB}}{\DHE{v\sB}},
\end{align*}
one can show that the interface component $u\sB$ 
of the 
solution to the IgA scheme
\equref{equ:ModelDisc} satisfies 
the variational identity
\begin{align}
\label{equ:VarSchur}
 s(u\sB,v\sB) = \func{g}{v\sB}\quad \forall v\sB\in \gVB,
\end{align}
where $g\in \gVB^*$ is a suitable functional.
By choosing the B-Spline basis for $\gVB$, the 
variational identity
\equref{equ:VarSchur} is equivalent to the linear system
\begin{align*}
 \Sb \uBb = \gb.
\end{align*}
The matrix $\Sb$ is the Schur complement matrix of $\Kb$ with respect to the interface dofs. 
Suppose, we reorder the entries of the stiffness matrix $\Kb$ and the load vector $\fvec$, such that the dofs corresponding to the interface come first, i.e.,
\begin{align*}
 \Kb = \MatTwo{\KBB}{\KBI}{\KIB}{\KII} \text{ and }\quad \fvec = \VecTwo{\fvecB}{\fvecI},
\end{align*}
then it can be shown that $\Sb$ and $\gb$ are given by
\begin{align*}
  \Sb &= \KBB - \KBI(\KII)^{-1}\KIB, \\
   \gb&= \fvecB -  \KBI (\KII)^{-1}\fvecI.
\end{align*}
Once $\uBb$ is calculated, we obtain $\uIb$ as the solution of the system
\begin{align*}
\KII \uIb = \fvecI - \KBI\uBb.
\end{align*}
Instead of the Schur complement matrix $\Sb$ we will mostly use its operator representation:
\begin{align*}
\Sop: \gVB \to \gVB^*,\\
 \func{\Sop v}{w} = \left(\Sb\vb,\wb\right). 
\end{align*}

%
\subsubsection{Approximation Properties}
This section recaps some important properties of the approximation power of B-Splines from \cite{veiga2006IGA_aopproximation}. First of all, we state a result about the relation of the $H^m$ norms between the function in the physical and the parameter domain, summarized in \propref{prop:HmGeoMapRelation} and \corref{cor:L2H1GeoMapRelation},
which are proved in \cite{veiga2006IGA_aopproximation}, {see  Lemma 3.5.}
\begin{proposition}
\label{prop:HmGeoMapRelation}
 Let $m$ be a non-negative integer, $\pCell\in\pMesh$ and $\gCell = \geoMap(\pCell)$. 
Then the equivalence inequalities
\begin{align}
  \HSNorm{m}{v}{\pCell} &\leq C_{shape} \LInfNorm{\det\nabla\geoMap^{-1}}{\gCell}^{1/2} \sum_{j=0}^m \LInfNorm{\nabla\geoMap}{\pCell}^j\HSNorm{j}{\g{v}}{\gCell},\\
  \HSNorm{m}{\g{v}}{\gCell} &\leq C_{shape} \LInfNorm{\det\nabla\geoMap}{\pCell}^{1/2}\LInfNorm{\nabla\geoMap}{\pCell}^{-m}\sum_{j=0}^m\HSNorm{j}{v}{\pCell}.
\end{align}
hold for all $\g{v}\in H^m(\gCell)$ and their counterparts $ v \in H^m({\hat Q})$,
where $C_{shape}$ are positive generic constants that only depend on the shape of the geometry $\gDom$ and its parametrization.
\end{proposition}
We note that the 0-order terms in the upper bounds of \propref{prop:HmGeoMapRelation} are not needed for $m>0$. They are incorporated in order to give a unified 
presentation
for $m\geq0$. 
Hence, 
as a special case of \propref{prop:HmGeoMapRelation}, we obtain the following estimates for the $L^2$ norm and $H^1$ seminorm
\begin{corollary}
\label{cor:L2H1GeoMapRelation}
 Let  $\pCell\in\pMesh$ and $\gCell = \geoMap(\pCell)$.\\
 For $\g{v}\in L^2(\gCell)$, we have
 \begin{align}
  \LTwoNorm{v}{\pCell} &\leq C_{shape} \LInfNorm{\det\nabla\geoMap^{-1}}{\gCell}^{1/2} \LTwoNorm{\g{v}}{\gCell},\\
  \LTwoNorm{\g{v}}{\gCell} &\leq C_{shape} \LInfNorm{\det\nabla\geoMap}{\pCell}^{1/2}\LTwoNorm{v}{\pCell},
 \end{align}
and, for $\g{v}\in H^1(\gCell)$, we can write
  \begin{align}
   \HOneSNorm{v}{\pCell} &\leq C_{shape}C \HOneSNorm{\g{v}}{\gCell},\\
   \HOneSNorm{\g{v}}{\gCell} & \leq C_{shape}C \HOneSNorm{v}{\pCell},
  \end{align}
  where the generic constants $C$  only depends  on the geometrical mapping.
\end{corollary}
The second result provides the quantitative approximation power of NURBS.
We mention that B-Splines are special case of NURBS. 
It basically states that a NURBS space has the same approximation power as a FE space of same degree. We do not give any further technical details. We refer the reader to   \cite{veiga2006IGA_aopproximation} 
for a more comprehensive discussion of the approximation properties.
In particular, the following theorem is proved in \cite{HL:TagliabueDedeQuarteroni:2014a},
see also Theorem~3.2 in \cite{veiga2006IGA_aopproximation}.


\begin{theorem}
\label{thm:approxPower}
Let $k$ and $l$  be integers with $0 \leq k\leq l \leq \deg+1$, $l\geq \mul$,
and  $\gCell = \geoMap(\pCell)$ with $\pCell \in \pMesh$. 
Then there exists a projective operator $\Pi_{\gVh}:L^2(\Omega) \to \gVh$ 
from $L^2(\Omega)$ to the NURBS space $\gVh$ such that
the approximation error estimates
\begin{align*}
 \sum_{\gCell\in\gMesh}\left|v-\Pi_{\gVh} v\right|^2_{H^k(\gCell)}  \leq C_{shape} \sum_{\gCell\in\gMesh} h_{\gCell}^{2(l-k)}\sum_{i=0}^l\LInfNorm{\nabla \geoMap}{\overline{\pCell}}^{2(i-l)}
&\HSNorm{i}{v}{\overline{\gCell}}^2 
\end{align*}
%
hold for all $v\in H^l({\gDom})$,
where $\overline{\gCell}$ denotes   the so-called support extension of $\gCell$,
and
the constant $C_{shape}$ only depends on the geometry and its parametrization.
%
\end{theorem}

\section{The IsogEometric Tearing and Interconnecting method and its Analysis}
\label{sec:IETI}
The IETI method, 
that was introduced in \cite{HL:KleissPechsteinJuettlerTomar:2012a},
is the adaption of the FETI method  (see, e.g., \cite{toselli2005DDM} or \cite{pechstein2012FETI}) 
to isogeometric analysis. 
According to \cite{ManDohrTez2005_AlgebTheoryFETIDP_BDDC} based on algebraic arguments, the BDDC preconditioner and the FETI-DP method  possess the same spectrum up to zeros and ones.
Hence a condition number bound for BDDC implies a bound for FETI-DP and vice versa. 
Since the proof is based on algebraic arguments, it also holds for the IETI-DP method. 
It was first introduced in \cite{HL:KleissPechsteinJuettlerTomar:2012a}, 
and further analysed in \cite{beirao2013IGA_bddc}.

In the following, let $\gVh$ be the 
conform IgA space
which fulfils the Dirichlet boundary conditions as defined in \secref{sec:DiscreteProblem}. Furthermore, we denote by $\gBSplFam$ the B-Spline basis of this space. 
The idea of IETI is to introduce local spaces, which are independent of each other, and consider only the local problems. The coupling and the continuity across interfaces is received via additional constraints.


\subsection{Local spaces and Jump operator}
Let $\gVh\sMP$ be the space of B-Spline functions defined on patch $\gDom\sMP$, see \equref{def:gVhMP}. Analogously to the splitting introduced in \secref{sec:IGA-disc}, we define the local interface space
\begin{align*}
 \LocVB[\iMP] := \spn\{\gBSpl |\, \supp\{\gBSpl\} \cap (\partial\gDom\sMP \cap \gIntMP) \neq \emptyset, i\in \iSet\},
\end{align*}
which is the restriction of $\gVB$ to $\gDom\sMP$. 
In the following, in order to avoid cumbersome notation,
 we define the patch boundary to be just the interface part, 
i.e. $\partial\gDom\sMP := \partial\gDom\sMP \cap \gIntMP$.  
Hence, we have
\begin{align*}
 \gVh\sMP = \LocVB[\iMP] \oplus \gVI\sMP,
\end{align*}
where $\gVI\sMP$ is defined as in \equref{equ:VIMP} or \equref{equ:VIMP2}.
Furthermore, we define the space of functions, which are locally in $\LocVB[\iMP]$, by
\begin{align*}
   \LocVB := \prod_{\iMP=1}^\nMP \LocVB[\iMP].
\end{align*}
We note that
functions 
from
$\LocVB$ are, in general, not continuous across the interface, i.e. $\LocVB \nsubseteq C^0$.
A function  $w \in \LocVB[\iMP]$ has components $ w := \left[w\sMP\right]_{\iMP=1}^\nMP \leftrightarrow \left[ \wb[\iMP] \right]_{\iMP=1}^\nMP =:\wb$.

In order to obtain continuous functions, we introduce additional constraints
which will enforce the continuity. Let $\IndexSet(k,l)$ be the set of all coupled indices between $\gDom\sMP[k]$ and $\gDom\sMP[l]$, then we enforce the following constraints
\begin{align}
  \label{equ:constrJumpop}
   \wb[k]_i - \wb[l]_j = 0 \qquad \forall (i,j) \in \IndexSet(k,l), k>l.
\end{align}
The operator $B : \LocVB \to \LamSet^*:=\real[\nLag]$, 
which realizes constraints \equref{equ:constrJumpop}
in the form
  \begin{align*}
    Bw = 0,
  \end{align*}
is called \emph{jump operator}.
The space of all functions in $\LocVB$
which belong to the kernel of $B$ 
is denoted by $\hW$, and can be identified with $\gVB$, i.e.
  \begin{align*}
   \hW = \{ w\in \LocVB|\, Bw =0 \} \equiv \gVB.
  \end{align*}

  \subsubsection{Saddle point formulation}
  Due to the multipatch structure of our physical domain, we can decompose the bilinear form and the right hand side functional as follows
  \begin{align*}
    \bil{u_h}{v_h} &= \sum_{\iMP=1}^\nMP \bil[\iMP]{u_h\sMP}{v_h\sMP},\\
    \func{F}{v_h} &= \sum_{\iMP=1}^\nMP \func[\iMP]{F}{v_h\sMP},
  \end{align*}
  where $u_h,v_h\in\gVh$ and according to \secref{sec:IGA-disc} $u_h\sMP,v_h\sMP$ denote its restriction to $\gDom\sMP$. By means of the B-Spline basis we can rewrite the variational problem as linear system
  \begin{align}
  \label{equ:decomVar}
    \left(\sum_{\iMP=1}^\nMP \mat{A}\Kb[\iMP]\mat{A}^T \right)\ub = \sum_{\iMP=1}^\nMP \mat{A}\fvec[\iMP],
  \end{align}
  where $\mat{A}$ is the Boolean patch assembling matrix.
  
  Analogously to \secref{sec:Schurcomp}, we can reorder the entries of the patch local stiffness matrix and right-hand side
  \begin{align*}
    \Kb[\iMP] = \MatTwo{\KBB[\iMP]}{\KBI[\iMP]}{\KIB[\iMP]}{\KII[\iMP]}, \qquad \fvec[\iMP] = \VecTwo{\fvecB[\iMP]}{\fvecI[\iMP]}.
  \end{align*}
  The equation $\Kb[\iMP] \ub[\iMP] = \fvec[\iMP]$ is equivalent to
  \begin{align*}
    \Sb[\iMP] \uBb[\iMP] =\gvec[\iMP],
  \end{align*}
  where $\Sb[\iMP] = \KBB[\iMP] - \KBI[\iMP](\KII[\iMP])^{-1}\KIB[\iMP]$  and $\gvec[\iMP]= \fvecB[\iMP] -  \KBI[\iMP] (\KII[\iMP])^{-1}\fvecI[\iMP]$.
 We obtain that equation \equref{equ:decomVar} can be reformulated as 
  \begin{align}
  \label{equ:decomSchurMat}
    \left(\sum_{\iMP=1}^\nMP \mat{A}\Sb[\iMP]\mat{A}^T \right) \uBb = \sum_{\iMP=1}^\nMP \mat{A}\gvec[\iMP].
  \end{align}
  Similarly, we can express \equref{equ:decomSchurMat} in operator notation as 
  \begin{align}
  \label{equ:decomSchurOp}
   \sum_{\iMP=1}^\nMP \func{\Sop[\iMP] u_{B,h}\sMP}{v_{B,h}\sMP} = \sum_{\iMP=1}^\nMP \func{g\sMP}{v_{B,h}\sMP} \quad \forall v_{B,h} \in \gVB,
  \end{align}
  where $u_{B,h}\in \gVB,\, g\sMP\in{\gVB\sMP}^*$ and $\Sop[\iMP]: \gVB\sMP \to {\gVB\sMP}^*$.

  Instead of defining the global Schur complement and the right-hand side functional on the conforming space $\gVB$, we define an \emph{extended} version of the Schur complement on the ``discontinuous'' space $\LocVB$, i.e.
   \begin{align*}
    \Sop:& \LocVB \to \LocVB^*\\ 
       \func{\Sop v}{w} &:= \sum_{\iMP=1}^\nMP\func{\Sop[\iMP] v\sMP}{w\sMP}  \quad \for v,w \in \LocVB,\\	      
    g&\in \LocVB^*\\
    \func{g}{w} &:= \sum_{\iMP=1}^\nMP\func{g\sMP}{w\sMP}  \quad \for w \in \LocVB.	
   \end{align*}
   Expressed in matrix form, we can write $\Sb$ and $\gvec$ as
   \begin{align*}
       \Sb:= \text{diag}(\Sb[\iMP])_{\iMP=1}^\nMP \quad \gb:=[\gvec[\iMP]]_{\iMP=1}^\nMP.
   \end{align*}
  
   The next step is to reformulate \equref{equ:decomSchurOp} in terms of $\Sop$ and $B$ in the space $\LocVB$. 
Due to the symmetry of $\bil{\cdot}{\cdot}$, we can write \equref{equ:decomSchurOp} as minimization problem
   \begin{align*}
    u_{B,h} = \underset{v\in \gVB}{\text{argmin}} \frac{1}{2} \sum_{\iMP=1}^\nMP \left(\func{\Sop[\iMP] v\sMP}{v\sMP} -  \func{g\sMP}{v\sMP}\right).
   \end{align*}
   It is easy to see that this problem is equivalent to
      \begin{align*}
    u_{B,h} = \underset{w\in \LocVB, Bw = 0}{\text{argmin}} \frac{1}{2} \func{\Sop w}{w} -  \func{g}{w}.
   \end{align*}
   In the following we will only work with the Schur complement system and hence, to simplify the notation, we will use $u$ instead of $u_{B,h}$, when we consider functions in $\gVB$. If there has to be made a distinction between $u_h, u_{B,h}$ and $u_{I,h}$, we will write the subscript letter.
   
This constraint minimization problem can be rewritten in form of  
the following saddle point problem:
find $(u,\lambdavec) \in \LocVB \times \LamSet$ such that
    \begin{align}
    \label{equ:saddlePointSing}
     \MatTwo{\Sop}{B^T}{B}{0} \VecTwo{u}{\lambdavec} = \VecTwo{g}{0},
    \end{align}
    \begin{theorem}
   If $\ker{\Sop} \cap \ker{B} = \{0\}$, then the above problem is uniquely solveable up to adding elements from $\ker{B^T}$ to $\lambdavec$. 
  \end{theorem}
\begin{remark}
We note that problem \equref{equ:saddlePointSing} is uniquely solvable 
with respect to $u$ (if $\Kop$ is regular). 
However, not all $\Sop[\iMP]$ are regular, since those, 
which do not lie on a Dirichlet boundary correspond to pure Neumann problems. 
The usual strategy for solving \equref{equ:saddlePointSing} is to work with its Schur complement $F$, but since some blocks of $\Sop$ are singular, the Schur complement is not well defined. In classical FETI/IETI, one adds the basis of each local kernel to the space and regularizes the matrix. Unfortunately, this need an exact knowledge of the kernels, which is in general not trivial. 
The dual primal approach presented below circumvents this by restricting the space $\LocVB$ in order to guarantee the regularity of each $\Sop[\iMP]$.
\end{remark}


\subsection{Intermediate space and primal constraints}

In order to guarantee the positive definiteness of $\Sop$, we are looking for an intermediate space $\tW$
in the sense $\hW\subset \tW \subset \LocVB$ such that $\Sop$ restricted to $\tW$ is SPD.
  Let $\Psi \subset \gVB^*$ be a set of linearly independent \emph{primal variables}.
Then we define the spaces
  \begin{align*}
    \tW &:= \{w\in \LocVB: \forall\psi \in \Psi: \psi(w\sMP[k]) = \psi(w\sMP[l]), \forall k>l  \},\\
    \Wd &:= \prod_{\iMP=1}^\nMP \Wd[\iMP], \quad \text{where } \Wd[\iMP]:=\{w\sMP\in \LocVB[\iMP]: \forall\psi \in \Psi: \psi(w\sMP[\iMP]) =0 \}.
  \end{align*}
  Moreover, we introduce the space $\Wp \subset \hW$, such that
  \begin{align*}
    \tW &= \Wp \oplus \Wd.
  \end{align*}
We call $\Wp$ \emph{primal space} and $\Wd$ \emph{dual space}. For the analysis, the subspace $\Wp$ is not required. However, it brings advantages in the implementation.
  
If we choose $\Psi$, such that $\tW \cap \ker{\Sop}=\{0\}$, then
  \begin{align*}
   \tSop: \tW \to \tW^* \quad \text{with } \func{\tSop v}{w} = \func{\Sop v}{w} \quad \forall v,w \in \tW
  \end{align*}
  is invertible. If a set $\Psi$ fulfils $\tW \cap \ker{\Sop} =\{0\}$, then we say that $\Psi$ \emph{controls the kernel} and in the following, we will always assume that such a set is chosen. 
  
In the literature,
there are the following typical choices for the primal variables $\psi$:
\begin{itemize}
    \item Vertex evaluation: $\psi^\mathcal{V}(v) = v(\mathcal{V})$,
    \item Edge averages: $\psi^\mathcal{E}(v) = \frac{1}{|\mathcal{E}|}\int_{\mathcal{E}}v\,ds$,
    \item Face averages: $\psi^\mathcal{F}(v) = \frac{1}{|\mathcal{F}|}\int_{\mathcal{F}}v\,ds$.
\end{itemize}
The typical choices for $\Psi$ are usually called Algorithm A - C:
\begin{itemize}
    \item Algorithm A: $\Psi = \{\psi^\mathcal{V}\}$,
    \item Algorithm B: $\Psi = \{\psi^\mathcal{V}\} \cup \{\psi^\mathcal{E}\} \cup \{\psi^\mathcal{F}\}$,
    \item Algorithm C: $\Psi = \{\psi^\mathcal{V}\} \cup \{\psi^\mathcal{E}\}$.
\end{itemize}
%
Moreover, one finds references to two further choices for $\Psi$, commonly referred to Algorithm D and E, which are aiming for a reduced set of primal variables, see, e.g. \cite{toselli2005DDM} Algorithm 6.28 and 6.29. This algorithms address the issue of the rapidly increasing number of primal variables.
\begin{remark}
For domains $\gDom\subset\real[2]$, Algorithm A will provide a quasi optimal method for the Poisson problem. By choosing additional primal variables, the coarse problem will grow. 
Hence, it becomes computationally more demanding. 
However, it brings benefits in the condition number. 
For three-dimensional domains, one can show that just choosing vertex evaluation does not lead to a quasi optimal method. In such cases, additional primal variables have to be chosen.
\end{remark}


\subsubsection{IETI - DP}
\label{sec:IETI_DP}
Since $\tW \subset \LocVB$, there is a natural embedding $\emb: \tW \to \LocVB$. 
Let the jump operator restricted to $\tW$ be
\begin{align}
   \label{def:Btilde}
    \tB := B\emb : \tW \to \LamSet^*.
\end{align}
Then we can formulate the saddle point problem in $\tW$ as follows:
find $(u,\lambdavec) \in \tW \times \LamSet:$
    \begin{align}
    \label{equ:saddlePointReg}
     \MatTwo{\tSop}{\tB^T}{\tB}{0} \VecTwo{u}{\lambdavec} = \VecTwo{\tg}{0},
    \end{align}
    where $\tg := \emb^T g$, and $\tB^T= \emb^T B^T$. Here, $\emb^T: \LocVB^* \to \tW^*$ denotes the adjoint of $\emb$, which can be seen as a partial assembling operator.

By construction, $\tSop$ is SPD on $\tW$. Hence, we can define the Schur complement $F$ 
and the corresponding right-hand side of equation \equref{equ:saddlePointReg} 
as follows:
\begin{align*}
    F&:= \tB \tSop^{-1}\tB^T,\\
    d&:= \tB\tSop^{-1} \tg.
\end{align*}
Hence, the saddle point system is equivalent to solving:
\begin{align}
   \label{equ:SchurFinal}
      \text{find } \lambdavec \in \LamSet: \quad F\lambdavec = d.
\end{align}
By means of 
Brezzi's theorem,
we obtain
  \begin{theorem}
   The above problem is uniquely solvable up to adding elements from $\ker{\tB^T}$ to $\lambdavec$. The unique solution
   \begin{align*}
      u = \tSop^{-1}(\tg-\tB^T\lambdavec)
   \end{align*}
   satisfies $u \in \hW \equiv \gVB$ and is the unique solution of \equref{equ:ModelDisc}.
  \end{theorem}
We note that $F$ is SPSD on $\LamSet$. 
According to \cite{pechstein2012FETI}, if we 
set
\begin{align*}
   \tLamSet &:= \LamSet_{/\ker{\tB^T}},\\
   \tLamSet^*&:= \range{\tB},
\end{align*}
 where $\tLamSet^*$ is in fact the dual of $\tLamSet$, 
then $F$ restricted to $\tLamSet^*$ is SPD, i.e. $F_{|\tLamSet}: \tLamSet\to\tLamSet^*$. 
Hence, it is possible to solve \equref{equ:SchurFinal} with PCG.


\subsubsection{BDDC - Preconditioner}
\label{sec:BDDC}
  Following the lines in \cite{pechstein2012FETI} a different but equivalent way 
  leads to
  the so called \emph{BDDC - Preconditioner}. 
Here  we start from the equation
  \begin{align}
    \label{equ:Schurhat}
   \hat{\Sop}\hat{u} = \hat{g},
  \end{align}
%
where the notation  $\hat{\cdot}$ indicates  
that the operator and the functions are restricted to the continuous space $\gVB$, i.e., just the standard Schur complement. 
Since $\hat{\Sop}$ is SPD, we can solve system \equref{equ:Schurhat} 
by the PCG method preconditioned by the so called BDDC preconditioner $M^{-1}_{BDDC}$.
  
As in the previous section, using a set of linearly independent primal variables $\Psi$ and the corresponding spaces, we define the BDDC preconditioner as follows
\begin{align*}
   M^{-1}_{BDDC} := \widetilde{E}_D \tSop^{-1} \widetilde{E}_D^T,
\end{align*}
where $\widetilde{E}_D^T$ is defined via the formulas
\begin{align*}
  \widetilde{E}_D &= E_D \widetilde{I}: \quad \tW\to\hW, \\\
  E_D &= I - P_D: \quad  \LocVB \to \hW, \\
  P_D &= B_D^T B: \quad\LocVB \to \LocVB.
\end{align*}
One can give an alternative formulation 
see, e.g., \cite{beirao2013IGA_bddc}. 
Here we assume, that after a change of basis, each primal variable corresponds to one degree of freedom, see, e.g., \cite{klawonnWidlung2006_FETIDP_linearElasticity}. Let $\Kb[\iMP] $ be the stiffness matrix corresponding to $\gDom\sMP$. We splitting the degrees of freedom into interior $(I)$ and interface $(B)$ ones. 
Furthermore, we again split the interface degrees of freedoms into primal $(\Pi)$ and dual $(\Delta)$ ones. 
This provides a partition of $\Kb[\iMP] $ into 2 x 2 and 3 x 3 
block systems:
\begin{align*}
 \Kb[\iMP]= \MatTwo{\KII[\iMP]}{\KIB[\iMP]}{\KBI[\iMP]}{\KBB[\iMP]}  = \MatThree{\KII[\iMP]}{\GenMatXX[\iMP]{\Kb}{I\Delta}}{\GenMatXX[\iMP]{\Kb}{I\Pi}}
 {\GenMatXX[\iMP]{\Kb}{\Delta I}}{\GenMatXX[\iMP]{\Kb}{\Delta \Delta}}{\GenMatXX[\iMP]{\Kb}{\Delta \Pi }}
 {\GenMatXX[\iMP]{\Kb}{\Pi I}}{\GenMatXX[\iMP]{\Kb}{\Pi \Delta}}{\GenMatXX[\iMP]{\Kb}{\Pi \Pi}}.
\end{align*}
In order to the define the preconditioner, we need the following restriction and interpolation operators:
\begin{align}
 \begin{split}
  \RestrBD:& \tW \to \Wd,\\
  \RestrBP:& \tW \to \Wp,\\
  \RestrD[\iMP]:& \Wd \to \Wd[\iMP],\\
  \RestrP[\iMP]:& \Wp \to \Wp[\iMP].
 \end{split}
\end{align}
We now derive a scaled version $\RestrScD[\iMP]$ from $\RestrD[\iMP]$ by multiplying its $i$-th row with $ {\delta^{\dagger}_i}\sMP$,
and then we define
\begin{align}
 \RestrScB:= \RestrBP \oplus \left(\sum_{\iMP=1}^{\nMP} \RestrScD[\iMP]\right)\RestrBD.
\end{align}
The BDDC preconditioner is 
now determined by 
\begin{align*}
 M^{-1}_{BDDC} := \RestrScB^T \tSop^{-1} \RestrScB,
\end{align*}
where
\begin{align*}
 \tSop^{-1} = \RestrBD^T\left(\sum_{\iMP=1}^{\nMP}\begin{bmatrix}0 & {\RestrD[\iMP]}^T\end{bmatrix} \MatTwo{\KII[\iMP]}{\GenMatXX[\iMP]{\Kb}{I\Delta}}{\GenMatXX[\iMP]{\Kb}{\Delta I}}{\GenMatXX[\iMP]{\Kb}{\Delta \Delta}}^{-1} \VecTwo{0}{\RestrD[\iMP]}\right)\RestrBD+ \Phi \SPP^{-1} \Phi^T.
\end{align*}
Here the matrices $\SPP$ and $\Phi$ are given by
\begin{align*}
 \SPP=\sum_{\iMP=1}^{\nMP}{\RestrP[\iMP]}^T\left(\GenMatXX[\iMP]{\Kb}{\Pi \Pi}- \begin{bmatrix}\GenMatXX[\iMP]{\Kb}{\Pi I} & \GenMatXX[\iMP]{\Kb}{\Pi \Delta}\end{bmatrix} \MatTwo{\KII[\iMP]}{\GenMatXX[\iMP]{\Kb}{I\Delta}}{\GenMatXX[\iMP]{\Kb}{\Delta I}}{\GenMatXX[\iMP]{\Kb}{\Delta \Delta}}^{-1} \VecTwo{\GenMatXX[\iMP]{\Kb}{ I \Pi}}{\GenMatXX[\iMP]{\Kb}{\Delta \Pi}}\right)\RestrP[\iMP]
\end{align*}
and 
%
\begin{align*}
 \Phi = \RestrBP^T - \RestrBD^T\sum_{\iMP=1}^{\nMP}\left(\begin{bmatrix}0 &  {\RestrD[\iMP]}^T\end{bmatrix} \MatTwo{\KII[\iMP]}{\GenMatXX[\iMP]{\Kb}{I\Delta}}{\GenMatXX[\iMP]{\Kb}{\Delta I}}{\GenMatXX[\iMP]{\Kb}{\Delta \Delta}}^{-1} \VecTwo{\GenMatXX[\iMP]{\Kb}{ I \Pi}}{\GenMatXX[\iMP]{\Kb}{\Delta \Pi}}\right)\RestrP[\iMP],
\end{align*}
respectively.
\subsection{Preconditioning}
\label{sec:IETI_precond}
The existent theory for FETI provides us with the following estimate for the condition number of $F$, see, e.g., \cite{pechstein2012FETI},
    \begin{theorem}
     Let $H\sMP$ be the diameter of patch $\iMP$. Then, under suitable assumptions imposed on the mesh, we have 
     \begin{align*}
       \kappa(F_{|\tLamSet}) \leq C \left(\max_\iMP\frac{H\sMP}{h\sMP}\right),
     \end{align*}
     where the constant positive $C$ is dependent on $\alpha$.
    \end{theorem}
%
Since this condition number bound is not optimal in comparison with other precondition strategies like multigrid, there is a need for additional preconditioning of $F$. 
This can be done by the so called  \emph{Dirichlet preconditioner} or its scaled versions. For finite elements, the Dirichlet preconditioner provides a quasi optimal condition number bound, which is not robust with respect to the diffusion coefficient. The \emph{Dirichlet preconditioner} $M_D^{-1}$ is defined as
      \begin{align*}
	M_D^{-1} = B \Sop B^T.
      \end{align*}
We note that this preconditioning just uses the block diagonal version of the Schur complement. Hence, the application can be done in parallel.

In order to receive robustness with respect to the diffusion coefficient $\alpha$, we use the scaled version the of Dirichlet preconditioner, the so called \emph{scaled Dirichlet preconditioner}. The scaling is incorporated in the application of the jump operator. Therefore, we define the \emph{scaled} jump operator $B_D$, such that the operator enforces the constraints:
    \begin{align*}
     &\Scal[j]\sMP[l]\wb[k]_{i} - \Scal[i]\sMP[k]\wb[l]_{j} = 0  \qquad \forall (i,j) \in \IndexSet(k,l), k>l, \\
      &\text{with } \Scal[i]\sMP = \frac{\rho\sMP_i}{\sum_{l} \rho\sMP[l]_{j_l}},
    \end{align*}
    where $j_l$ is the corresponding coefficient index on the neighbouring patch $\gDom\sMP[l]$. The scaled Dirichlet preconditioner has the following form
   \begin{align}
   \label{equ:scaled_Dirichlet}
	M_{sD}^{-1} = B_D \Sop B_D^T.
   \end{align}
   Typical choices for $\rho\sMP_i$ are
       \begin{itemize}
     \item Multiplicity Scaling: $\rho\sMP_i = 1$,
     \item Coefficient Scaling: If $\alpha(x)_{|\gDom\sMP} = \alpha\sMP$, choose $\rho\sMP_i = \alpha\sMP$, 
     \item Stiffness Scaling: $\rho\sMP_i = \Kb\sMP_{i,i}$. 
    \end{itemize}
  If the diffusion coefficient $\alpha$ is constant and identical on each patch, then the multiplicity and the coefficient scaling are the same. If there is only a little variation in $\alpha$, then the multiplicity scaling provides good results. If the variation is too large, one should use the other scalings to obtain robustness. 
    \begin{theorem}
     Let $H\sMP$ be the diameter and $h\sMP$ the local mesh size of $\gDom\sMP$ and let $M_{sD}^{-1}$ be the scaled Dirichlet preconditioner.
     
     Then, under suitable assumption imposed on the mesh, we have 
     \begin{align*}
       \kappa(M_ {sD}^{-1} F_{\tLamSet}) \leq C \max_\iMP\left(1+\log\left(\frac{H\sMP}{h\sMP}\right)\right)^2,
     \end{align*}
     where the positive constant $C$ is independent of $h$ and $H$.
    \end{theorem}
   In the case of IgA, a more general proof in the sense, that not only $C^0$ smoothness across patch interfaces is allowed but also $C^l,\,l\geq 0$ smoothness, can be found in \cite{beirao2013IGA_bddc}. However, the proof is restricted to the case of a domain decomposition,
   which is obtained by subdividing a single patch, i.e. performing a decomposition of the parameter domain. Hence, always the same geometrical mapping $\geoMap$ is used. Furthermore, due to the $C^l,\,l\geq 0$, smoothness across interfaces, only a condition number bound of $O\left((1+\log H/h)H/h\right)$ could be proven for stiffness scaling. In the proceeding section, we will extend the proof given in \cite{beirao2013IGA_bddc} to multipatch domains, which consists of different geometrical mappings $\geoMap\sMP$ for each patch. Additionally, for $l=0$, we again obtain quasi-optimal condition number bounds also for stiffness scaling.
 \subsection{Analysis of BDDC-Preconditiner}
 \label{sec:AnalysisBDDC}
 In this section we rephrase the results and notations established in  \cite{beirao2013IGA_bddc}, and extend them to multipatch domains, consisting of a different geometrical mapping $\geoMap\sMP$ for each patch. However, we only allow $C^0$ smoothness across the patch interfaces.
 \subsubsection{General Results}
 Let $\g{z}$ be a function from $\gVh$. Its restriction to a patch $\gDom\sMP$ belongs to $\gVh\sMP$, and can be written as
 \begin{align*}
   \g{z}\sMP(\x) := \g{z}(\x)_{|\gDom\sMP} = \sum_{i\in\iSet\sMP} c_i\sMP \gBSpl\sMP(\kn),
 \end{align*}
 where in $\iSet\sMP$ all indices, where the basis functions have a support on the Dirichlet boundary in the physical space $\gDom\sMP$, are excluded. The corresponding spline function in the parameter space is denoted by $z\sMP(\kn)\in\pVh\sMP$.It is important to note, that the geometrical map $\geoMap$ and its inverse $\geoMap^{-1}$ are independent of $h$, since it is fixed on a coarse discretization. When the domain becomes refined, $\geoMap$ stays the same. Clearly, the same applies for the gradients and it can be assumed, that $\geoMap\sMP \in W^{1,\infty}(\pDom) $ for all $\iMP \in\{1,\ldots,\nMP\}$.
 
 Let $\pCell\in\pMesh\sMP$ be a cell in the parameter domain $\pDom$, and $\gCell\in\gMesh\sMP$ be a cell in the physical domain $\gDom\sMP$. Then we denote all indices of basis functions, which have a support on $\pCell$ and $\gCell$, respectively, by 
 \begin{align}
  \iSet\sMP(\pCell)&:=\{i\;|\;\pCell\subseteq\supp\{\BSpl\sMP\}\} =\{i\;|\;\gCell\subseteq\supp\{\gBSpl\sMP\}\}.
 \end{align}
 
 We will now define a local discrete norm based on control points $c_i$.
 \begin{definition}
 Let $\pCell\in\pMesh$, $\gCell = \geoMap\sMP(\pCell)$, $\g{z}\in \gVh\sMP$ and $z$ its counterpart in the parameter domain. We define
 \begin{align*}
  \DiscLLocNorm{z}{\pCell}^2 := \left(\max_{i\in \iSet\sMP(\pCell)}|c_i\sMP|^2\right)h_{\pCell}^2,\\
  \DiscLLocNorm{\g{z}}{\gCell}^2 := \left(\max_{i\in \iSet\sMP(\pCell)}|c_i\sMP|^2\right)h_{\gCell}^2,
 \end{align*}
 where  $h_{\gCell} =\LInfNorm{\nabla\geoMap}{\pCell} h_{\pCell}$.
 \end{definition}
 \begin{remark}
 Note that the two discrete norms are obviously equivalent, since
  \begin{align}
  \begin{split}
  \DiscLLocNorm{z}{\pCell}^2 &= \left(\max_{i\in \iSet\sMP(\pCell)}|c_i|^2\right)h_{\pCell}^2 = \left(\max_{i\in \iSet\sMP(\pCell)}|c_i|^2\right) \LInfNorm{\nabla\geoMap}{\pCell}^{-2} h_{\gCell}^2\\
  &= \LInfNorm{\nabla\geoMap}{\pCell}^{-2}\DiscLLocNorm{\g{z}}{\gCell}^2.
  \end{split}
 \end{align}
 \end{remark}
 \begin{lemma}
 \label{lem:EquDiscL2Loc}
Let $\pCell\in\pMesh$, $\gCell = \geoMap\sMP(\pCell)$, $\g{z}\in \gVh\sMP$ and $z$ its counterpart in the parameter domain, then
 \begin{align}
  \LTwoNorm{z}{\pCell} \approx \DiscLLocNorm{z}{\pCell}.
 \end{align}
 \end{lemma}
 \begin{proof}
  See \cite{beirao2013IGA_bddc}.
 \hfill$\square$\end{proof}
  From \lemref{lem:EquDiscL2Loc} and \corref{cor:L2H1GeoMapRelation}, we obtain
  \begin{corollary}
  \label{cor:EquDiscL2}
   Let $\pCell\in\pMesh$, $\gCell = \geoMap\sMP(\pCell)$, $\g{z}\in \gVh\sMP$ and $z$ its counterpart in the parameter domain. Then we have
   \begin{align}
   \label{equ:equv_disc_L2}
   \LTwoNorm{\g{z}}{\gCell}\approx \DiscLLocNorm{\g{z}}{\gCell}^2 .
   \end{align}
  \end{corollary}
   \begin{proof}
   Indeed, it is easy to see that
   \begin{align*}
    \LTwoNorm{\g{z}}{\gCell}^2 &\leq \LInfNorm{\det\nabla\geoMap}{\pCell} \LTwoNorm{z}{\pCell}^2 \approx \LInfNorm{\det\nabla\geoMap}{\pCell}\DiscLLocNorm{z\sMP}{\pCell}^2,
    \end{align*}
    and
    \begin{align*}
    \DiscLLocNorm{z\sMP}{\pCell}^2 &\approx \LTwoNorm{z}{\pCell}^2 \leq \LInfNorm{\det\nabla\geoMap^{-1}}{\gCell} \LTwoNorm{\g{z}}{\gCell}^2,
    \end{align*}
    which yield \equref{equ:equv_disc_L2}.
   \hfill$\square$\end{proof}
  It immediately follows that the norm in the physical space $\LTwoNorm{\g{z}}{\gCell}$ is also equivalent to $\DiscLLocNorm{z}{\pCell}^2$, since
  $\LTwoNorm{\g{z}}{\gCell} \approx \DiscLLocNorm{\g{z}}{\gCell}^2 \approx \DiscLLocNorm{z}{\pCell}^2$. Now we define a global patchwise norm, which will be again equivalent to the $L^2$-norm in the parameter domain.
   \begin{definition}
      Let $\g{z}\in \gVh\sMP$ and $z$ its counterpart in the parameter domain. The global discrete norm is defined by
      \begin{align}
       \DiscLNorm{z}^2:= \sum_{i\in\iSet\sMP}|c_i\sMP|^2 h_{\pCell}^2.
      \end{align}
   \end{definition}
  \begin{lemma}
  Let $\g{z}\in \gVh\sMP$ and $z$ its counterpart in the parameter domain. Then we have the norm equivalence
  \begin{align}
    \DiscLNorm{z}^2 \approx \LTwoNorm{z}{\pDom}^2.
  \end{align}
  \end{lemma}
   \begin{proof}
  See \cite{beirao2013IGA_bddc}.
 \hfill$\square$\end{proof}
 Again, it follows that the norm is also equivalent to the corresponding one in the physical space.
 \begin{corollary}
    Let $\g{z}\in \gVh\sMP$ and $z$ its counterpart in the parameter domain. Then we have the norm equivalence
   \begin{align*}
   \LTwoNorm{\g{z}}{\gDom\sMP}\approx \DiscLNorm{z}^2 .
   \end{align*}
  \end{corollary}
   \begin{proof}
  The proof is identical to that of \corref{cor:EquDiscL2} with the same constants.
   \hfill$\square$\end{proof}
   The next step is to define a discrete $H^1$ seminorm based on the coefficients $c_i$. We denote by $c_{i,i^\iVar-j}$ the coefficient corresponding to basis function $\BSplG{(i^1,\ldots,i^\iVar-j,\ldots,i^d)}{\deg}$, see \equref{def:TBSpline}. The derivative with respect to $\kn^\iVar$ of a B-Spline $z$ can be written as
   \begin{align}
    \dd{z}{\kn^\iVar} = \sum_{i\in\iSet_\iVar}\left(\frac{c_{i,i^\iVar}-c_{i,i^\iVar-1}}{\Delta_i^\iVar}\right)\BSplG{i}{(\deg,\deg^\iVar-1)},
   \end{align}
   where $\iSet_\iVar$ is the set of admissible indices such that each summand is well defined, $\Delta_i^\iVar = \kn[i+\deg]^\iVar - \kn[i]^\iVar$, and  $\BSplG{i}{(\deg,\deg^\iVar-1)}$ is a tensor B-Spline of degree $\deg$, where its degree in dimension $\iVar$ is reduced by one.

   With this definition at hand, we are able to define a discrete seminorm $\DiscLLocNorm{\cdot}{\kn^\iVar,\pCell}$ and $\DiscLLocNorm{\cdot}{\kn^\iVar}$ in the parameter domain  as follows:
   \begin{definition}
    Let $\pCell\in\pMesh$, $\gCell = \geoMap\sMP(\pCell)$, $\g{z}\in \gVh\sMP$, and $z$ its counterpart in the parameter domain. Then we define the local discrete seminorm by
 \begin{align}
    \DiscLLocNorm{z}{\kn^\iVar,\pCell}^2:= \max_{i\in\iSet_\iVar\sMP(\pCell)}|c_{i,i^\iVar}\sMP -c_{i,i^\iVar-1}\sMP|^2,
   \end{align}
  where $\iSet_\iVar\sMP(\pCell):=\{i\in \iSet\sMP(\pCell)\;|\;\pCell \subseteq \supp\{\BSplG{i}{(\deg,\deg^\iVar-1)}\}$.
   The global patchwise counterpart is then defined by
   \begin{align}
    \DiscLLocNorm{z}{\kn^\iVar}^2:= \sum_{i\in\iSet_\iVar\sMP}|c_{i,i^\iVar}\sMP -c_{i,i^\iVar-1}\sMP|^2.
   \end{align}
   \end{definition}
   This seminorms are again equivalent to the standard Sobolev seminorms.
   \begin{lemma}
   \label{lem:EquDiscH1Loc}
    Let $\pCell\in\pMesh$, $\gCell = \geoMap\sMP(\pCell)$, $\g{z}\in \gVh\sMP$ and $z$ its counterpart in the parameter domain. Then we have the equivalence
    \begin{align}
     \DiscLLocNorm{z}{\kn^\iVar,\pCell}^2 \approx \LTwoNorm{\dd{z}{\kn^\iVar}}{\pCell}^2.
    \end{align}
   \end{lemma}
   \begin{proof}
   See \cite{beirao2013IGA_bddc}.
  \hfill$\square$\end{proof}
%
%
   We are now in the position to define the ``full'' discrete seminorm
   \begin{definition}
    Let $\g{z}\in \gVh\sMP$, and $z$ its counterpart in the parameter domain. Then we define the ``full'' discrete seminorm by the formula
    \begin{align}
     \DiscHNorm{z}^2:=\sum_{\iVar=1}^\pDim  \DiscLLocNorm{z}{\kn^\iVar}^2.
    \end{align}
   \end{definition}
   \lemref{lem:EquDiscH1Loc} and \corref{cor:L2H1GeoMapRelation} immediately yield the following proposition.
   \begin{proposition}
   \label{lem:EquDiscH1}
    Let $\g{z}\in \gVh\sMP$ and $z$ its counterpart in the parameter domain. Then the following equivalences hold:
    \begin{align}
     \DiscHNorm{z}^2 \approx \HOneSNorm{z}{\pDom}^2 \approx \HOneSNorm{\g{z}}{\gDom\sMP}^2.
    \end{align}
   \end{proposition}
  The next step is to provide properties in the local index spaces. Since we consider only the two dimensional problem, we can interpret the control points $(c_i)_{i\in\iSet}$ as entries of a matrix $\mat{C}=(c_i)_{i=1}^{\nBasis_1,\nBasis_2}$. Before doing that we will provide abstract results for an arbitrary matrix.
  
  We define the seminorm
  \begin{align}
   \RealSNorm{\mat{C}}^2:= \sum_{\iVar=1}^2 \sum_{\substack{i=1 \\i^\iVar=2}}^{\nBasis_\iVar}|c_{i,i^\iVar}-c_{i,i^\iVar-1}|^2
  \end{align}
  for a real valued $\nBasis_1\times \nBasis_2$ matrix $\mat{C}=(c_i)_{i=1}^{\nBasis_1,\nBasis_2}$ $\in \real[{\nBasis_1\times \nBasis_2}]$. The entries of the matrix $\mat{C}$ can be interpreted as values on a uniform grid $\mathcal{T}$. This motivates the definition of an operator, which evaluates a continuous function on the grid points $(x_i) =(x_{i^1 i^2})$:
  \begin{align}
   \begin{split}
      (\cdot)_I:&\, C(\pDom[2]) \to \real[{\nBasis_1\times \nBasis_2}],\\
	      & f\mapsto f_I: (f_I)_{i} = f(x_{i}),
   \end{split}
  \end{align}
  and an operator, that provides a piecewise bilinear interpolation of the given grid values
    \begin{align}
   \begin{split}
      \chi:& \real[{\nBasis_1\times \nBasis_2}] \to \mathcal{Q}_1(\mathcal{T}) \subset H^1(\pDom[2]),\\
	      & \mat{C}\mapsto \chi(\mat{C}): \chi(\mat{C})(x_{i}) = c_{i},
   \end{split}
  \end{align}
  where $\mathcal{Q}_1(\mathcal{T})$ is the space of piecewise bilinear functions on $\mathcal{T}$.

  Furthermore, given values on an edge on $\pDom[2]$, we need to define its linear interpolation and a discrete harmonic extension to the interior. In order to do so, let $e$ be an edge on $\pDom[2]$, and let us denote  all indices of grid points $x_i$ associated to $e$ by $\iSet(e)$. Additionally, let $\mathcal{P}_1(\mathcal{T}_{|e})$ be the space of piecewise linear spline functions on $\mathcal{T}_{|e}$.
  
  We define the interpolation of values on $\iSet(e)$ by the restriction of the operator $\chi$ to $e$:
  \begin{align*}
  \chi_e &: \real[\nBasis_\iVar] \to H^1(e)\\
   &\vec{v} \mapsto \chi_e(\vec{v})\in \mathcal{P}_1(\mathcal{T}_{|e}): \chi_e(\vec{v})(x_{i}) = v_{i^\iVar}, \quad i\in\iSet(e), 1\leq i^\iVar\leq \nBasis_\iVar,
  \end{align*}
  where $e$ corresponds to dimension $\iVar$, $\iVar\in\{1,2\}$.
  
  This leads to a definition of a seminorm for grid points on an edge $e$ via the interpolation to functions from  $\mathcal{P}_1(\mathcal{T}_{|e})$:
  \begin{definition}
   Let $e$ be an edge of $\pDom[2]$ along dimension $\iVar$ and $\vec{v}$ be a vector in $\real[\nBasis_\iVar]$. Then we define the following seminorm
   \begin{align}
     \RealTNorm{\vec{v}} := \left|\chi_e(\vec{v})\right|_{H^{1/2}(e)}, \quad \forall \vec{v}\in\real[\nBasis_\iVar].
   \end{align}
  \end{definition}
  It remains to define the interpolation operator for the whole boundary, which will be the interpolation of all four edges.
  \begin{definition}
   Let
  \begin{align}
   \iSet(\partial): = \{i: x_i \in \partial\pDom[2]\}.
  \end{align}
  be all indices $i$ such that $x_i$ lies on the boundary $\partial\pDom[2]$. For $\vec{b}\in \real[|\iSet(\partial)|]$, let $\vec{b}_{|e}$ be all the values which are associated with $e$. The interpolation operator $\chi_\partial$ is then defined by
  \begin{align}
  \begin{split}
   \chi_\partial&: \real[|\iSet(\partial)|] \to H^1(\partial\pDom[2])\\
	  &\vec{b} \mapsto \chi_\partial(\vec{b}): \chi_\partial(\vec{b})_{|e} = \chi_e(\vec{b}_{|e}).
  \end{split}
  \end{align}
  \end{definition}
  \begin{remark}
   For the interpolation operator $\chi_e$ defined on an edge $e$, the equations 
     $\chi(\mat{C})_{|e} = \chi_e(\mat{C}_{|e})$ and 
     $\RealTNorm{\mat{C}_{|e}}= \HSNorm{1/2}{\chi_e(\mat{C}_{|e})}{e}=\HSNorm{1/2}{\chi(\mat{C})_{|e}}{e}$ are obviously valid.
  \end{remark}
  Finally, we are able to define the discrete harmonic extension in $\real[\nBasis_1\times\nBasis_2]$.
  \begin{definition}
   Let $\DHEreal$ be the standard discrete harmonic extension into the piecewise bilinear space $\mathcal{Q}_1$. This defines the lifting operator $\mat{H}$ by
   \begin{align}
   \begin{split}
    \mat{H} &: \real[|\iSet(\partial)|] \to \real[\nBasis_1\times\nBasis_2]\\
             &\vec{b}\mapsto \mat{H}(\vec{b}):= (\DHEreal[\chi_\partial(\vec{b})])_I.
   \end{split}
   \end{align}
  \end{definition}
  \begin{theorem}
  \label{thm:realMatProp}
   Let $e$ be a particular side on the boundary of $\pDom[2]$ and the constant $\beta\in\real[+]$ such that
   \begin{align*}
    \beta^{-1}\nBasis_2 \leq \nBasis_1\leq \beta\nBasis_2.
   \end{align*}
   Then the following hold:
   \begin{itemize}
    \item For all $\vec{b}\in \real[2\nBasis_1 + 2\nBasis_2 -1]$ that vanish on the four components corresponding to the four corners, it holds 
    \begin{align*}
     \RealSNorm{\mat{H}(\vec{b})}^2 \leq c (1+\log^2\nBasis_1) \sum_{e\in\partial\pDom[2]}\RealTNorm{\vec{b}_{|e}}^2,
    \end{align*}
    where the constant c depends only on $\beta$.
    \item It holds 
    \begin{align*}
      \RealSNorm{\mat{C}}\geq c \RealTNorm{\mat{C}_{|e}},\quad \forall \mat{C} \in\real[{\nBasis_1\times\nBasis_2}],
    \end{align*}
    where the constant c depends only on $\beta$.
   \end{itemize}
  \end{theorem}
  \begin{proof}
   See \cite{beirao2013IGA_bddc}.
  \hfill$\square$\end{proof}
%
\subsubsection{Condition number estimate}
The goal of this section to establish a condition number bound for $P = M_{BDDC}^{-1}\hat{\Sop}$. 
Following \cite{beirao2013IGA_bddc}, we assume  that the mesh is quasi-uniform on each subdomain and the diffusion coefficient
is globally constant.
We focus now on a single patch $\gDom\sMP$, $\iMP\in\{1,\ldots,\nMP\}$.  
For notational simplicity, we assume that the considered patch $\gDom\sMP$ does not touch the boundary $\partial\gDom$. 

We define the four edges of the parameter domain $(0,1)^2$ by $E_r$, and their images by $\g{E}_r = \geoMap\sMP(E_r)$, $r=1,2,3,4$.
In order to represent functions $\g{w}\sMP\in\LocVB[\iMP]$ as the corresponding vectors, we need the coefficients corresponding to the basis functions on the four edges,  namely
\begin{align}
\begin{split}
 \iSet(\g{E}_1\sMP) &= \{i\;|\;i^1 = 1,\, i^2 = 1,2,\ldots\nBasis_2\sMP\},\\
 \iSet(\g{E}_2\sMP) &= \{i\;|\;i^1 = \nBasis_1\sMP,\, i^2 = 1,2,\ldots\nBasis_2\sMP\},\\
 \iSet(\g{E}_3\sMP) &= \{i\;|\;i^1 = 1,2,\ldots\nBasis_1\sMP, \, i^2 = 1\},\\
 \iSet(\g{E}_4\sMP) &= \{i\;|\;i^1 = 1,2,\ldots\nBasis_1\sMP, \, i^2 = \nBasis_2\sMP\},\\
 \iSet(\gIntMP\sMP) &= \bigcup_{1\leq r\leq 4} \iSet(\g{E}_r\sMP).
 \end{split}
\end{align}
Let $\g{z}\sMP\in\gVh\sMP$, i.e., $\g{z}\sMP  = \sum_{i\in\iSet\sMP}c_i^z\gBSpl\sMP.$
%
Hence, $\g{z}\sMP$ is determined by its coefficients $c_i^z = (c^z_{i^1,i^2})_{i^1,i^2 = 1}^{\nBasis_1\sMP,\nBasis_2\sMP}$, which can be interpreted as a $\nBasis_1\sMP\times\nBasis_2\sMP$ matrix $\mat{C}^z$, i.e.,
\begin{align*}
 \gVh\sMP\ni\g{z}\sMP = \sum_{i\in\iSet\sMP}c_i^z\gBSpl\sMP \Longleftrightarrow \{c_i^z\}_{i\in\iSet\sMP} \Longleftrightarrow \mat{C}^z = (c^z_{i^1,i^2})\in\real[{\nBasis_1\sMP\times\nBasis_2\sMP}].
\end{align*}
Moreover, we can identify functions on the trace space $\LocVB[\iMP]$ as follows:
\begin{align*}
 \LocVB[\iMP]\ni\g{w}\sMP &= \sum_{i\in\iSet(\gIntMP\sMP)}c_i^w\gBSpl\sMP \Longleftrightarrow \{c_i^w\}_{i\in\iSet\sMP(\gIntMP\sMP)},\\
  \LocVB[\iMP]_{|E_r\sMP}\ni\g{w}\sMP &= \sum_{i\in\iSet(\g{E}_r\sMP)}c_i^w\gBSpl\sMP \Longleftrightarrow \{c_i^w\}_{i\in\iSet(\g{E}_r\sMP)} \quad \text{for }r=1,2,3,4.
\end{align*}
Similar to the definition of the global discrete NURBS harmonic extension $\DHE$ in \equref{def:DHE}, we define the local patch  version 
\begin{align}
 \DHE\sMP : \LocVB[\iMP]\to \gVh\sMP
\end{align}
such that, for any $\g{w}\sMP\in \LocVB[\iMP]$, $\g{z}\sMP:= \DHE\sMP\left(\g{w}\sMP\right)$ is the unique function in $\gVh$ which minimizes the $H^1$ energy on $\gDom\sMP$ and
 $ c_i^z = c_i^w\quad \forall i\in\iSet(\gIntMP\sMP).$
%
%
%
Finally, let $\Wd[\iMP]\subset\LocVB[\iMP]$ be the space of spline functions which vanish on the primal variables, i.e., in the corner points. The following theorem provides an abstract estimate of the condition number using the coefficient scaling: 
\begin{theorem}
\label{thm:abstractBDDC}
 Let the counting function $\Scal\sMP$ be chosen accordingly to the coefficient scaling strategy. Assume that there exist two positive constants $c_*,c^*$ and a boundary seminorm $\DiscLLocNorm{\cdot}{\LocVB[\iMP]}$ on $\LocVB[\iMP]$, $\iMP=1,\ldots,\nMP$, such that
 \begin{align}
  \label{equ:abstractBDDC_1}
  \DiscLLocNorm{\g{w}\sMP}{\LocVB[\iMP]}^2 &\leq c^* s\sMP(\g{w}\sMP,\g{w}\sMP) \quad\forall \g{w}\sMP\in \LocVB[\iMP],\\
  \label{equ:abstractBDDC_2}
  \DiscLLocNorm{\g{w}\sMP}{\LocVB[\iMP]}^2 &\geq c_* s\sMP(\g{w}\sMP,\g{w}\sMP) \quad\forall \g{w}\sMP\in \Wd[\iMP],\\
  \label{equ:abstractBDDC_3}
  \DiscLLocNorm{\g{w}\sMP}{\LocVB[\iMP]}^2 &= \sum_{r=1}^4\DiscLLocNorm{\g{w}\sMP_{|E_r\sMP}}{\LocVB[\iMP]_r}\quad\forall \g{w}\sMP\in \LocVB[\iMP],
 \end{align}
 where $\DiscLLocNorm{\g{w}\sMP_{|E_r\sMP}}{\LocVB[\iMP]_r}$ is a seminorm associated to the edge spaces $\LocVB[\iMP]_{|E_r\sMP}$, with $r=1,2,3,4$. Then the condition number of the preconditioned BDDC operator P satisfies the bound
 \begin{align}
  \kappa(M_{BDDC}^{-1}\hat{\Sop})\leq C(1+c_*^{-1} c^*),
 \end{align}
where the constant $C$ is independent of $h$ and $H$.
\end{theorem}
\begin{proof}
 See \cite{beirao2013IGA_bddc} or \cite{beirao2010_BDDC_FEM_Plates}.
\hfill$\square$\end{proof}
Using this abstract framework, we obtain  the following condition number estimate for the BDDC preconditioner.
\begin{theorem}
\label{thm:CondCoeffScal}
 There exists a boundary seminorm such that the constants $c_*$ and $c^*$ of \thmref{thm:abstractBDDC} are bounded by 
 \begin{align}
  c^* &\leq C_1,\\
  c_*^{-1} &\leq C_2 \max_{1\leq\iMP\leq\nMP}\left(1+\log^2\left(\frac{H\sMP}{h\sMP}\right)\right),
 \end{align}
 where the constants $C_1$ and $C_2$ are independent of $H$ and$ h$. Therefore, the condition number of the isogeometric preconditioned BDDC operator is bounded by
 %
 \begin{align}
  \kappa(M_{BDDC}^{-1}\hat{\Sop}) \leq C \max_{1\leq\iMP\leq\nMP}\left(1+\log^2\left(\frac{H\sMP}{h\sMP}\right)\right),
 \end{align}
 where the constant $C$ is independent of $H$ and $h$.
\end{theorem}
\begin{proof}
 The proof essentially follows the lines of the proof given in \cite{beirao2013IGA_bddc} with a minor modification due to the different geometrical mappings $\geoMap\sMP$. We note that we only consider $C^0$ continuity across the patch interfaces, which makes the proof less technical.
 
 The first step is to appropriately define the seminorm $\DiscLLocNorm{\g{w}\sMP}{\LocVB[\iMP]}^2$ in $\LocVB[\iMP]$:
 \begin{align}
 \label{def:seminormLocVB}
 \begin{split}
  \DiscLLocNorm{\g{w}\sMP}{\LocVB[\iMP]}^2&:= \sum_{r=1}^4\DiscLLocNorm{\g{w}\sMP_{|E_r\sMP}}{\LocVB[\iMP]_r},\\
  \DiscLLocNorm{\g{w}\sMP_{|E_1\sMP}}{\LocVB[\iMP]_1} &:= \RealTNorm{\g{w}\sMP_{|E_1\sMP}}^2 + \sum_{i^2=1}^{\nBasis_2{\sMP}-1}|c^w_{(1,i^2+1)}-c^w_{(1,i^2)}|^2,\\
  \DiscLLocNorm{\g{w}\sMP_{|E_2\sMP}}{\LocVB[\iMP]_2} &:= \RealTNorm{\g{w}\sMP_{|E_2\sMP}}^2 + \sum_{i^2=1}^{\nBasis_2{\sMP}-1}|c^w_{(\nBasis_1,i^2+1)}-c^w_{(\nBasis_1,i^2)}|^2,\\
  \DiscLLocNorm{\g{w}\sMP_{|E_3\sMP}}{\LocVB[\iMP]_3} &:= \RealTNorm{\g{w}\sMP_{|E_3\sMP}}^2 + \sum_{i^1=1}^{\nBasis_1{\sMP}-1}|c^w_{(i^1+1,1)}-c^w_{(i^1,1)}|^2,\\
  \DiscLLocNorm{\g{w}\sMP_{|E_4\sMP}}{\LocVB[\iMP]_4} &:= \RealTNorm{\g{w}\sMP_{|E_4\sMP}}^2 + \sum_{i^1=1}^{\nBasis_1{\sMP}-1}|c^w_{(i^1+1,\nBasis_2)}-c^w_{(i^1,\nBasis_2)}|^2,
  \end{split}
  \end{align}
where $\nBasis_\iVar\sMP$ denotes the number of basis functions on patch $\iMP$ in direction $\iVar$. Furthermore, we define $\RealTNorm{\g{w}\sMP_{|E_r\sMP}}^2:= \RealTNorm{\vec{v}}$, where $\vec{v}$  are the values $(c_i^w)_{i\in \iSet(\g{E}_r\sMP)}$ written as a vector.

Let $\g{z}\sMP\in \gVh\sMP$ be the NURBS harmonic extension of $w\sMP = \{c_i^w\}\in\LocVB[\iMP]$, and $z\sMP$ its representation in the parameter domain. Additionally, let $e$ be any edge of the parameter domain of $\gDom\sMP$. Due to the fact that 
\begin{align*}
                      c_i^w = c_i^z \quad \forall i\in \iSet(\gIntMP\sMP),
\end{align*}
and denoting $\mat{C\sMP} = (c_i^z)_{i\in\iSet\sMP}$, we obtain 
%
 $\RealTNorm{\g{w}_{|e}\sMP}^2 = \RealTNorm{\mat{C}_{|e}\sMP}^2\leq c \RealSNorm{\mat{C}\sMP}^2$ 
by means of \thmref{thm:realMatProp}.
From the definition of $\RealSNorm{\mat{C}\sMP}^2$ and the definition of $\DiscLLocNorm{\g{w}\sMP_{|E_r\sMP}}{\LocVB[\iMP]_r}^2$, we get
\begin{align}
 \DiscLLocNorm{\g{w}\sMP_{|e}}{\LocVB[\iMP]_r} ^2 \leq c \RealSNorm{\mat{C}\sMP}^2.
\end{align}
Furthermore, we have
\begin{align*}
 \DiscLLocNorm{\g{w}\sMP_{|e}}{\LocVB[\iMP]_r}^2 \leq c \RealSNorm{\mat{C}\sMP}^2 \leq c \DiscHNorm{z\sMP}^2 \leq c \HOneSNorm{z\sMP}{\pDom} \leq c \HOneSNorm{\g{z}\sMP}{\gDom\sMP}.
\end{align*}
Since
%
 $ \HOneSNorm{\g{z}\sMP}{\gDom\sMP} = \HOneSNorm{\DHE\sMP\left(\g{w}\sMP\right)}{\gDom\sMP}^2 = s\sMP(\g{w}\sMP,\g{w}\sMP),$
we arrive at the estimate 
\begin{align*}
 \DiscLLocNorm{\g{w}\sMP_{|e}}{\LocVB[\iMP]_r}^2 \leq c \,s\sMP(\g{w}\sMP,\g{w}\sMP).
\end{align*}
This estimates hold for all edges of $\gDom\sMP$. Hence, it follows that
\begin{align*}
 \DiscLLocNorm{\g{w}\sMP}{\LocVB[\iMP]}^2\leq c^* s\sMP(\g{w}\sMP,\g{w}\sMP) \quad \forall \g w \in \LocVB[\iMP],
\end{align*}
where the constant does not depend on $h$ and $H$. This proofs the upper bound of the estimate, i.e., \equref{equ:abstractBDDC_1}.

Let be $\g{w}\sMP\in \Wd[k]$, $w\sMP$ its representation in the parameter domain, and $(c_i^w)_{i\in \iSet(\gIntMP\sMP)}$ its coefficient representation. We now apply the lifting operator $\mat{H}\sMP$ to $(c_i^w)_{i\in \iSet(\gIntMP\sMP)}$, and obtain a matrix with entries $\mat{H}\sMP(w\sMP) = (h_i\sMP)_{i\in\iSet\sMP}$. These entries define a spline function 
\begin{align*}
  z\sMP &:= \sum_{i\in\iSet\sMP}c_i^z \BSpl\sMP\\
  c_i^z &:= h_i\sMP \quad \forall i\in \iSet\sMP.
  \end{align*}
Now we obtain the estimate
  \begin{align*}
   \RealSNorm{\mat{H}\sMP(w\sMP)} ^2 = \DiscHNorm{z\sMP}^2 \geq c \HOneSNorm{z\sMP}{\pDom}^2\geq c \HOneSNorm{\g{z}\sMP}{\gDom\sMP}^2 \geq c\HOneSNorm{\DHE[\g{w}\sMP]}{\gDom\sMP}^2,
  \end{align*}
where the last inequality holds, due to the face that the discrete NURBS harmonic extension minimizes the energy among functions with given boundary data $\g{w}$. The constant $c $ does not depend on $h$ or $H$. 

Recalling the definition of $\DiscLLocNorm{\g{w}\sMP}{\LocVB[\iMP]}^2$ and using \thmref{thm:realMatProp}, we arrive at the estimates
\begin{align*}
  \RealSNorm{\mat{H}\sMP(w\sMP)} ^2 \leq c \left(1+\log^2 \nBasis\sMP\right)\sum_{e\in\partial\pDom[2]}\RealTNorm{w_{|e}\sMP}^2 \leq c \left(1+\log^2 \nBasis\sMP\right)\DiscLLocNorm{\g{w}\sMP}{\LocVB[\iMP]}^2.
\end{align*}
Due to the mesh regularity, we have $\nBasis\sMP \approx H\sMP/h\sMP$, and, hence, we obtain
\begin{align*}
 s\sMP(\g{w}\sMP,\g{w}\sMP) &= \HOneSNorm{\DHE[\g{z}\sMP]}{\gDom\sMP}^2 \leq c \left(1+\log^2 \left(\frac{H\sMP}{h\sMP}\right)\right)\DiscLLocNorm{\g{w}\sMP}{\LocVB[\iMP]}^2,
\end{align*}
which provides
\begin{align*}
  c_* &\leq c\max_{1\leq\iMP\leq\nMP}\left(1+\log^2\left(\frac{H\sMP}{h\sMP}\right)\right),
\end{align*}
where the constant $c$ is again independent of $h$ and $H$.
\hfill$\square$\end{proof}
The next theorem provides a the corresponding estimates for the stiffness scaling.
\begin{theorem}
\label{thm:CondStiffScal}
 Let the counting functions be chosen according to the stiffness scaling strategy. Assume that there exist two positive constants $c_*,c^*$ and a boundary seminorm $\DiscLLocNorm{\cdot}{\LocVB[\iMP]}$ on $\LocVB[\iMP]$, $\iMP=1,\ldots,\nMP$, such that the three conditions of \thmref{thm:abstractBDDC} hold. Moreover, we assume that it exits a constant $c_{STIFF}^*$ such that
 \begin{align}
  \label{equ:est_cSTIFF}
  \DiscLLocNorm{\g{w}\sMP}{\LocVB[\iMP]}\leq c_{STIFF}^* s(\delta \g{w}\sMP,\delta \g{w}\sMP) \quad \forall \g{w}\sMP\in\Wd[\iMP],
 \end{align}
 where the coefficients of $\delta \g{w}\sMP$ are given by  $c_i\sMP  \delta_i\sMP$.
%
Then the condition number of the preconditioned BDDC operator $M_{BDDC}^{-1}\hat{\Sop}$ satisfies the bound
\begin{align}
  \kappa(M_{BDDC}^{-1}\hat{\Sop}) \leq c (1+ c_*^{-1} c^* + c_*^{-1}c^*_{STIFF})
\end{align}
for some constant $c$ which is independent of $h$ and $H$.
\end{theorem}

\begin{proof}
 See \cite{beirao2013IGA_bddc}.
\hfill$\square$\end{proof}

According to \cite{beirao2013IGA_bddc}, we apply a modified version of the stiffness scaling where we use one representative of the values $\delta\sMP_i$. This is reasonable, since these values are very similar on one patch $\delta\sMP_i \approx \delta\sMP_j$, which arises from the tensor product structure of B-Splines and the constant material value on a patch.
\begin{lemma}
 The bound \equref{equ:est_cSTIFF} holds with 
 \begin{align}
  c^*_{STIFF} \leq 1.
 \end{align}
Hence, the condition number of the BDDC preconditioned system in the case of stiffness scaling is bounded by
\begin{align}
  \kappa(M_{BDDC}^{-1}\hat{\Sop})\leq C\max_{1\leq\iMP\leq\nMP}\left(1+\log^2\left(\frac{H\sMP}{h\sMP}\right)\right),
\end{align}
where the constant $C$ is independent of $H$ and $h$.
\end{lemma}
\begin{proof}
 The inequality
 \begin{align*}
   \DiscLLocNorm{\g{w}\sMP}{\LocVB[\iMP]}\leq c_{STIFF}^* s(\delta \g{w}\sMP,\delta \g{w}\sMP) \quad \forall \g{w}\in\Wd[\iMP]
  \end{align*}
 is equivalent to
 \begin{align}
   \DiscLLocNorm{\Scal\g{w}\sMP}{\LocVB[\iMP]}\leq c_{STIFF}^* s(\g{w}\sMP,\g{w}\sMP) \quad \forall \g{w}\in\Wd[\iMP].
 \end{align}
We  have already proven that 
\begin{align*}
  \DiscLLocNorm{\g{w}\sMP}{\LocVB[\iMP]}\leq c_{STIFF}^* s(\g{ w}\sMP,\g{w}\sMP) \quad \forall \g{w}\in\Wd[\iMP] \subset \LocVB[\iMP].
\end{align*}
Hence, it is enough to show the inequality
\begin{align*}
 \DiscLLocNorm{\Scal\g{w}\sMP}{\LocVB[\iMP]}\leq c_{h,H} \DiscLLocNorm{\g{w}\sMP}{\LocVB[\iMP]}\quad \forall \g{w}\in\Wd[\iMP],
\end{align*}
where the constant $c_{h,H}$ may depend on $h\sMP$ and $H\sMP$.
%
Recalling the definition of $\DiscLLocNorm{\Scal\g{w}\sMP}{\LocVB[\iMP]}$ as the sum of $\DiscLLocNorm{\Scal\g{w}\sMP_{|E_r\sMP}}{\LocVB[\iMP]_r}$, $r=1,2,3,4$, see \equref{def:seminormLocVB}, we have only to estimate the parts, e.g., $\DiscLLocNorm{\g{w}\sMP_{|E_1\sMP}}{\LocVB[\iMP]_1}$.
The other three terms follow analogously.
From the fact that $\Scal[i]\sMP\leq 1$ and $\Scal[i]\sMP=\Scal[i+1]\sMP$ for all $\iMP\in\{1,\ldots,\nMP\}$ and $i\in\iSet(\g{E}_1\sMP)$, c.f. Section\,6.2. in \cite{beirao2013IGA_bddc} , it follows that
\begin{align*}
\RealTNorm{\Scal\g{w}\sMP_{|E_1\sMP}} = \Scal\RealTNorm{\g{w}\sMP_{|E_1\sMP}}\leq\RealTNorm{\g{w}\sMP_{|E_1\sMP}}
\end{align*}
and
\begin{align*}
 \sum_{i^2=1}^{\nBasis_2{\sMP}-1}| \Scal[(1,i^2+1)] c^w_{(1,i^2+1)}-\Scal[(1,i^2)] c^w_{(1,i^2)}|^2 &= \sum_{i^2=1}^{\nBasis_2{\sMP}-1}\Scal[i]|  c^w_{(1,i^2+1)}- c^w_{(1,i^2)}|^2\\
  &\leq \sum_{i^2=1}^{\nBasis_2\sMP-1}|  c^w_{(1,i^2+1)}- c^w_{(1,i^2)}|^2.
\end{align*}
These estimates provide the inequalities
\begin{align}
 \DiscLLocNorm{\Scal\g{w}\sMP_{|E_1\sMP}}{\LocVB[\iMP]_1} \leq \DiscLLocNorm{\g{w}\sMP_{|E_1\sMP}}{\LocVB[\iMP]_1},
\end{align}
and, finally, 
\begin{align}
\DiscLLocNorm{\Scal\g{w}\sMP}{\LocVB[\iMP]} \leq\DiscLLocNorm{\g{w}\sMP}{\LocVB[\iMP]}.
\end{align}
This concludes the proof with $c_{STIFF}^* \leq 1$, and the desired condition number bound.
\hfill$\square$\end{proof}

  \section{Implementation}
  \label{sec:implementation}
  Since $\Fb$ is symmetric and at least positive semi definite and positive definite on $\tLamSet$, we can solve
   the linear system $\Fb \lambdavec = \db$ of the algebraic equations by means of the PCG algorithm, where we use $M_{sD}^{-1}$ as preconditioner:
   \begin{algorithm}
  \caption{PCG method}
  \label{alg:PCG}
   \begin{algorithmic}
    \State $\lambdavec_0$ given
    \State $r_0 = \db-\Fb\lambdavec_0, \quad k = 0,\quad \beta_{-1} = 0$
    \Repeat
      \State $s_k = M_ {sD}^{-1}r_k$
      \State $\beta_{k-1} = \frac{(r_k,s_k)}{(r_{k-1},s_{k-1})}$
      \State $p_k = s_k + \beta_{k-1}p_{k-1}$
      \State $\alpha_k = \frac{(r_k,s_k)}{(\Fb p_k,p_k)}$
      \State $\lambdavec_{k+1} = r_k + \alpha_k p_k$
      \State $r_{k+1} = r_k - \alpha_k \Fb p_k$
      \State $k = k+1$
    \Until{stopping criterion fulfilled}
   \end{algorithmic}
   \end{algorithm}

   It is very expensive to build up the matrices $\Fb$ and $M_ {sD}^{-1}$. Fortunately, the PCG algorithm only requires the application of the matrix on a vector. Hence, we want to present a way to efficiently apply $\Fb$ and $M_ {sD}^{-1}$ without calculating its matrix form. In order to do so, we need the following realization:
   \begin{itemize}
    \item Application of $\tSop^{-1}: \tW^* \to \tW$ and $\Sop: \LocVB \to \LocVB^*$,
    \item Realization and a basis of $\tW = \tWp \oplus \prod \tWd[\iMP]$,
    \item Representation of $w\in \tW$ as $\{\wP,\{\wD[\iMP]\}_\iMP\}$,
    \item Representation of $f\in \tW^*$ as $\{\fP,\{\fD[\iMP]\}_\iMP\}$,
    \item Application of $\tB$ and $\tB^T$.
   \end{itemize}
%

  \subsection{Choosing a basis for $\tWp$}
  The first step is to provide an appropriate space $\tWp$ and a basis $\{\tphi_j\}_j^{\np}$ for $\tWp$, where $\np$ is the dimension of $\tWp$, i.e., the number of primal variables. We request from the basis that it has to be nodal with respect to the primal variables, i.e.,
  \begin{align*}
     \psi_i(\tphi_j) = \delta_{i,j}, \quad \forall i,j \in\{1,\ldots,\np\}.
  \end{align*}
  Additionally, we require that 
  \begin{align*}
   \tphi_{j|\gDom\sMP} = 0 \quad \text{if }\psi_j\text{ is not associated to }\gDom\sMP,
  \end{align*}
  i.e., the basis has a local support in a certain sense.
  
 There are many choices for the subspace $\tWp$.  
 Following the approach presented in \cite{pechstein2012FETI}, we will choose that one which is orthogonal to $\tWd$ with respect to $\Sop$.Such a subspace exits since $\tSop$ is SPD on $\tW$. Hence, we can define $\tWp:=\tWd^{\perp_S}$, i.e.,
 \begin{align*}
  \func{\Sop w_\Pi}{\wD} = 0, \quad \forall w_\Pi \in \tWp,\wD \in \tWd. 
\end{align*}
This choice, which will simplify the application of $\tSop^{-1}$ significantly, is known as \emph{energy minimizing primal subspace} in the literature, cf. \cite{pechstein2012FETI},\cite{dohrmann2003_MinEnergy}.

Now we need to find a nodal local basis for this space. In order to do that, we define the constraint matrix $\Ci[\iMP]: \LocVB[\iMP]\to \real[{\np[\iMP]}]$ for each patch $\gDom\sMP$ which realizes the primal variables:
   \begin{align*}
     \Ci[\iMP]&: \LocVB[\iMP]\to \real[{\np[\iMP]}]\\
	    & (\Ci[\iMP] v)_j = \psi_{i(\iMP,j)}(v) \quad  \forall v\in \LocVB\, \forall j\in\{1,\ldots,\np[\iMP]\},
   \end{align*}
   where $\np[\iMP]$ are the number of primal variables associated with $\gDom\sMP$ and $i(\iMP,j)$ the global index of the $j$-th primal variable on $\gDom\sMP$. Note that a function $\wD\sMP\in \tWd[\iMP]$ is in the kernel of $\Ci[\iMP]$, i.e., $\Ci[\iMP]\wD\sMP = 0$. 

   For each patch $\iMP$, the basis functions $\{\tphi_j\sMP\}_{j=1}^{\np[\iMP]}$ of $\tWp\sMP$ are the solution of the system
    \begin{align}
    \label{equ:SC_basis}
     \MatTwo{\Sop[\iMP]}{{\Ci[\iMP]}^T}{\Ci[\iMP]}{0}\VecTwo{\tphi_j\sMP}{\tmu_j\sMP} = \VecTwo{0}{\evec_j\sMP}, \quad \forall j\in\{1,\ldots,\np[\iMP]\},
    \end{align}
    where $\evec_j\sMP \in \real[{\np[\iMP]}]$ is the $j$-th unit vector. The Lagrange multipliers $\tmu_j\sMP \in \real[{\np[\iMP]}]$ will become important later on. Equation \equref{equ:SC_basis} has a unique solution $\{\tphi_j\sMP,\tmu_j\sMP\}$. Due to  $\ker{\Sop} \cap \tW = \{0\}$ and $\ker{\Ci[\iMP]} = \Wd[\iMP]$, it follows that $\ker{\Sop[\iMP]}\cap\ker{\Ci[\iMP]} = \{0\}$. Furthermore, $\Psi$ is linearly independent, and, hence,  $\ker{{\Ci[\iMP]}^T} = \{0\}$. Both properties provide the existence and uniqueness of the solution. 
    
   Due to the fact that building the Schur complement $\Sop[\iMP]$ is not efficient, we use an equivalent formulation by means of $\Kop[\iMP]$:
     \begin{align}
     \label{equ:KC_basis}
     \MatThree{\KBBop[\iMP]}{\KBIop[\iMP]}{{\Ci[\iMP]}^T}{\KIBop[\iMP]}{\KIIop[\iMP]}{0}{\Ci[\iMP]}{0}{0}\VecThree{\tphi_j\sMP}{\cdot}{\tmu_j\sMP} = \VecThree{0}{0}{\evec_j\sMP}.
    \end{align}
    For each patch $\iMP$, the LU factorization of this matrix is computed and stored.

\subsection{Application of $\tSop^{-1}$}
Assume that $f:=\{\fP,\{\fD[\iMP]\}\} \in \tW^*$ is given. We are now looking for $w:=\{\wP,\{\wD[\iMP]\}\}\in \tW$ with
        \begin{align}
       w = \Sop^{-1}f.
     \end{align}
    The idea to provide a formula for the application of $\tSop^{-1}$ is via a block $LDL^T$ factorization of $\tSop$ with respect to the splitting of $\tW = \tWp \oplus \prod \tWd[\iMP]$. Let $\SPPop,\SDPop, \SPDop, \SDDop$ be the restrictions of $\Sop$ to the corresponding subspaces, i.e.,
    \begin{align*}
     \func{\SPPop v_\Pi}{w_\Pi} &= \func{\Sop v_\Pi}{w_\Pi} \quad &\for v_\Pi, w_\Pi \in \tWp,\\
     \func{\SPDop v_\Pi}{w_D} &= \func{\Sop v_\Pi}{w_D} \quad &\for v_\Pi \in \tWp,  w_D\in \tWd,\\
     \func{\SDDop v_D}{w_D} &= \func{\Sop v_D}{w_D} \quad &\for v_D, w_D \in \tWd,
    \end{align*}
   and $\SDPop = \SPDop^T$. We note that $\SDDop$ can be seen as a block diagonal operator, i.e., $\SDDop = \text{diag}(\SDDop[\iMP])$. Based on this splitting, we have the block form
   \begin{align*}
    \tSop &= \MatTwo{\SPPop}{\SPDop}{\SDPop}{\SDDop}.
    \end{align*}
    A formal block $LDL^T$ factorization of $\Sop$ yields
    \begin{align}
    \label{equ:tS_inv}
    \tSop^{-1} &= \MatTwo{\Ib}{0}{-\SDDop^{-1}\SDPop}{\Ib}\MatTwo{\boldsymbol{T}_{\Pi}^{-1}}{0}{0}{\SDDop^{-1}}\MatTwo{\Ib}{-\SPDop\SDDop^{-1}}{0}{\Ib},\\
     \label{equ:T_inv}
    \boldsymbol{T}_{\Pi}^{-1}&= \SPPop - \SPDop\SDDop^{-1}\SDPop.
   \end{align}

    Due to our special choice $\tWp:=\tWd^{\perp_S}$, we have $\SDPop = \SPDop = 0$. This simplifies the expression \equref{equ:tS_inv} and \equref{equ:T_inv}. Hence, we obtain
     \begin{align*}
      \tSop^{-1} = \MatTwo{\SPPop^{-1} }{0}{0}{\SDDop^{-1}}.
     \end{align*}
     Therefore, the application of $\tSop^{-1}$ reduces to an application of one global coarse problem involving $\SPP^{-1}$ and $\nMP$ local problems involving ${\SDDop[\iMP]}^{-1}$
     \begin{align*}
      \wP &= \SPP^{-1} \fP, \\
      \wD[\iMP] &= {\SDDop[\iMP]}^{-1}\fD[\iMP] \quad \forall \iMP=1,\ldots,\nMP .
     \end{align*}

\paragraph{Application of ${\SDDop[\iMP]}^{-1}:$}
\hfill
\\
     The application of ${\SDDop[\iMP]}^{-1}$ corresponds to solving a local Neumann problem in the space $\tWd$, i.e.,
     \begin{align*}
       \Sop[\iMP] w\sMP= \fD[\iMP],
     \end{align*}
     with the constraint $\Ci[\iMP] w\sMP = 0$.
     
     This problem can be rewritten as a saddle point problem in the form
     \begin{align*}
     \label{equ:SC_solution}
      \MatTwo{\Sop[\iMP]}{{\Ci[\iMP]}^T}{\Ci[\iMP]}{0}\VecTwo{w\sMP}{\cdot} = \VecTwo{\fD[\iMP]}{0}.
     \end{align*}
     The same method as used in \equref{equ:SC_basis} for rewriting this equation in terms of $\Kop$ applies here and the LU factorization of the matrix is already available.
  \paragraph{Application of ${\SPP[\iMP]}^{-1}:$}
  \hfill
  \\
     The matrix $\SPP$ can be assembled from the patch local matrices $\SPP[\iMP]$.
     
     Let $\{\tphi_j\sMP\}_{j=1}^{\np[\iMP]}$ be the basis of $\tWp[\iMP]$.  In order to assemble $\SPP[\iMP]$, in general, we have to compute
     \begin{align*}
      {\left(\SPP[\iMP]\right)}_{i,j} = \func{\Sop[\iMP] \tphi_i\sMP}{\tphi_j\sMP}, \quad i,j\in\{1,\ldots,\np[\iMP]\}.
     \end{align*}
     The special choice of $\tWp:=\tWd^{\perp_S}$ provides us with the following property
     \begin{align*}
      {\left(\SPP[\iMP]\right)}_{i,j} &= \func{\Sop[\iMP] \tphi_i\sMP}{\tphi_j\sMP} = -\func{{\Ci[\iMP]}^T \tmu_i\sMP}{\tphi_j\sMP} = -\func{\tmu_i\sMP}{\Ci[\iMP] \tphi_j\sMP}\\
				  & = -\func{\tmu_i\sMP}{\evec_{j}\sMP}= - \left(\tmu_i\sMP\right)_j.
     \end{align*}
     Therefore, we can reuse the Lagrange multipliers $\tmu_i\sMP$ obtained in \equref{equ:KC_basis}, and can assemble $\SPP\sMP$ from them. Once $\SPP$ is assembled, the  LU factorization can be calculated and stored.
   \subsection{Summary of the algorithm  for $F = \tB \Sop^{-1} \tB^T$}
   \begin{algorithm}
   \caption{Algorithm for the calculation of $\nu = F\lambdavec$ for given $\lambdavec \in \LamSet$}
   \label{alg:applyF}
   \begin{algorithmic}
    \State Application of $B^T:$ $\{f\sMP\}_{\iMP=1}^\nMP = B^T\lambdavec$
     \State Application of $\emb^T:$ $\{\fP,\{\fD[\iMP]\}_{\iMP=1}^\nMP\} = \emb^T\left(\{f\sMP\}_{\iMP=1}^\nMP\right)$
    \State  Application of $\Sop^{-1}:$
	  \begin{itemize}
	   \item  $\wP = \SPP^{-1} \fP$
	   \item  $\wD[\iMP] = {\SDDop[\iMP]}^{-1}\fD[\iMP] \quad \forall \iMP=1,\ldots,\nMP$ 
	  \end{itemize}
   \State  Application of $\emb:$ $\{w\sMP\}_{\iMP=1}^\nMP = \emb\left(\{\wP,\{\wD[\iMP]\}_{\iMP=1}^\nMP\} \right)$
   \State  Application of $B:$ $ \nu = B\left( \{w\sMP\}_{\iMP=1}^\nMP \right)$
  \end{algorithmic}
  \end{algorithm}
   Let us mention that the implementation of the embedding operator $\emb$ and assembling operator $\emb^T$ is explained in detail in \cite{pechstein2012FETI} and is omitted here. 
  \subsection{Application of the preconditioner}
    The application of the preconditioner $M_{sD}^{-1} = B_D \Sop B_D^T$ is basically the application of $\Sop$:
    \begin{align*}
     \Sop &=  \text{diag}(\Sop[\iMP]),\\
     \Sop[\iMP] &= \KBBop[\iMP] - \KBIop[\iMP](\KIIop[\iMP])^{-1}\KIBop[\iMP].
    \end{align*}
    The calculation of $v\sMP = \Sop[\iMP]w\sMP$ consists of 2 steps:
    \begin{enumerate}
     \item Solve: $\KIIop[\iMP] x\sMP = -\KIBop[\iMP]w\sMP\quad$ (\emph{Dirichlet problem}).
     \item $v\sMP = \KBBop[\iMP] w\sMP + \KBIop[\iMP]x\sMP$.
    \end{enumerate}   
     Again, a LU factorization of $\KIIop[\iMP]$ can be computed in advance and stored.

\section{Numerical examples}
\label{sec:numExamplesIETI}
We test the implemented IETI-DP algorithm for solving large scale systems arising from the IgA discretization of \equref{equ:ModelStrong} on the domains illustrated in \figref{fig:YETI_footprint}. The computational domain consists of 21 subdomains in both 2D and 3D. In both cases, one side of a patch boundary has inhomogeneous Dirichlet conditions, whereas all other sides have homogeneous Neumann conditions. Each subdomain has a diameter of $H$ and an associated mesh size of $h$. The degree of the B-Splines is chosen as $\deg = 4$.  In order to solve the linear system \equref{equ:SchurFinal}, a PCG algorithm with the scaled Dirichlet preconditioner \equref{equ:scaled_Dirichlet} is performed. We use zero initial guess, and  a reduction of the initial residual by a factor of $10^{-6}$ as stopping criterion. The numerical examples illustrate the dependence of the condition number of the IETI-DP preconditioned system on jumps in the diffusion coefficient $\alpha$, patch size $H$, mesh size $h$ and the degree $\deg$.

We use the C++ library G+SMO\footnote{\url{https://ricamsvn.ricam.oeaw.ac.at/trac/gismo/wiki/WikiStart}}, for describing the geometry and performing the numerical tests.
 \begin{figure}[h!]
  \begin{subfigmatrix}{2}
\subfigure[2D YETI-footprint]{\includegraphics[width=0.45\textwidth]{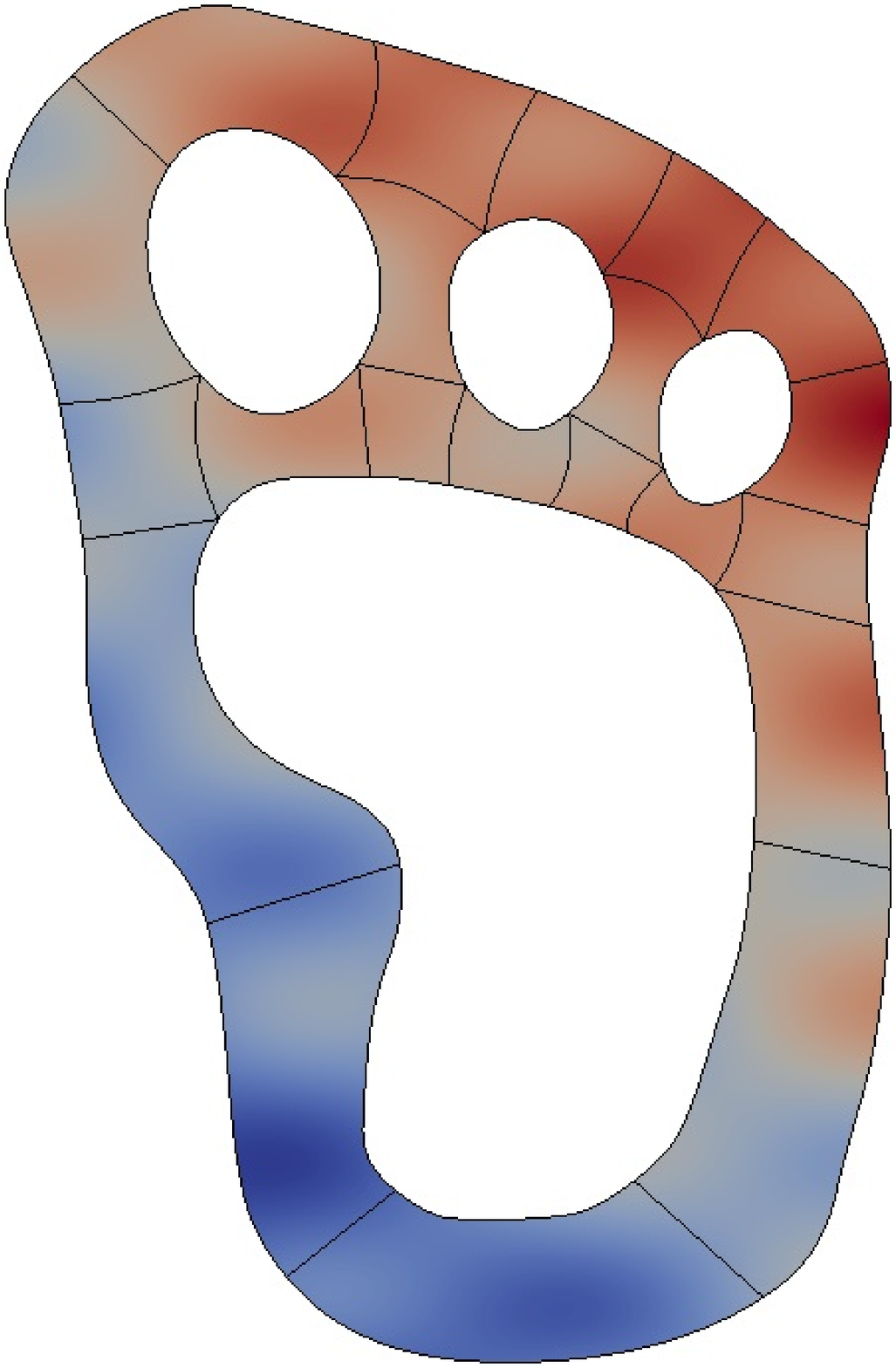} }
\subfigure[3D YETI-footprint]{\includegraphics[width=0.45\textwidth]{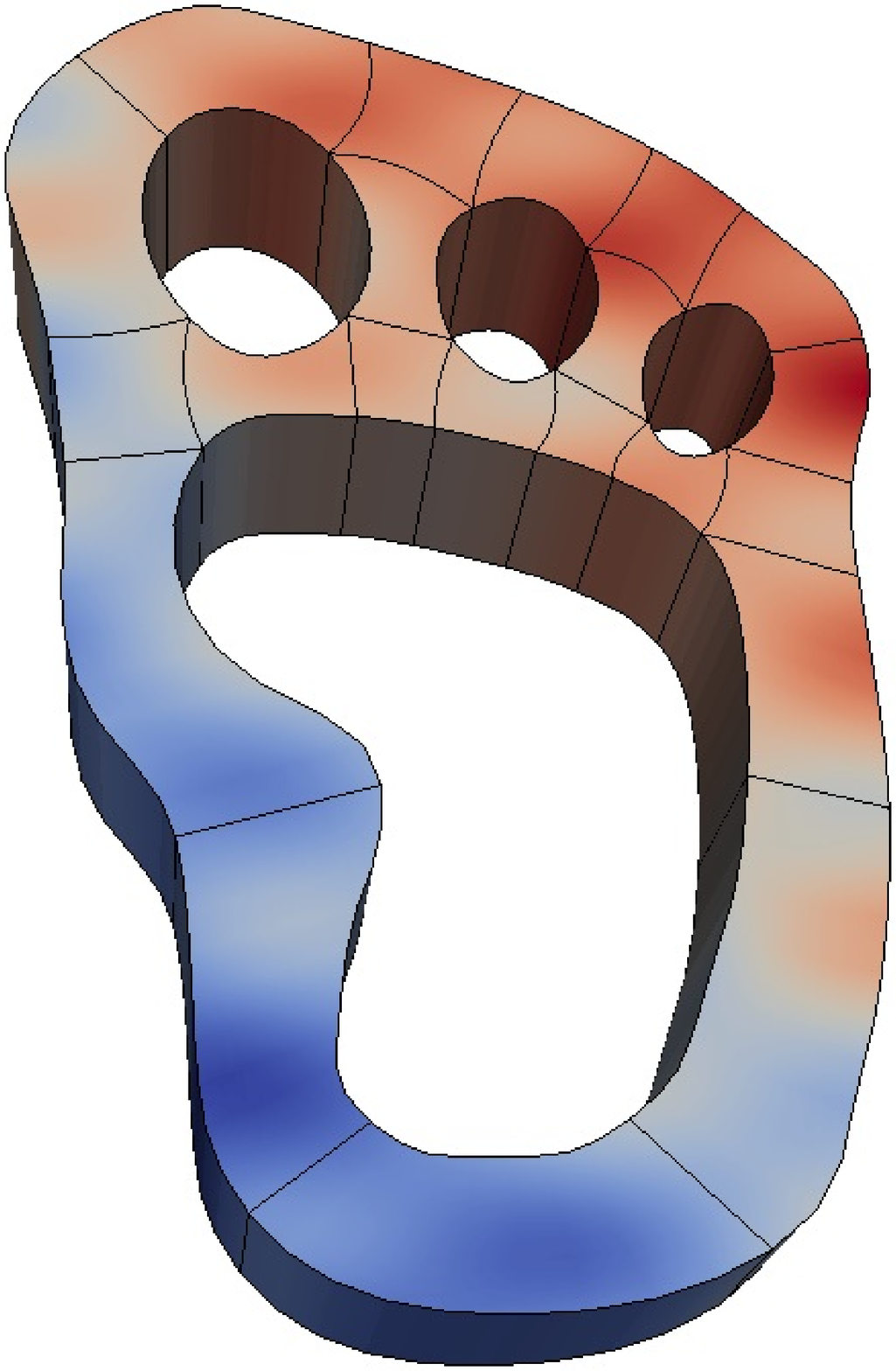} }
 \end{subfigmatrix}
  \caption{Illustration of the computational domain in 2D and 3D.}
  \label{fig:YETI_footprint}
  \end{figure}
  
\subsection{The case of homogeneous diffusion coefficient}
We present numerical tests for problem \equref{equ:ModelStrong} with globally constant diffusion coefficient $\alpha$. The 2D results are summarized in \tabref{table:hom2D}, whereas the 3D results are presented in  \tabref{table:hom3D}. The results confirm that the preconditioned systems with coefficient scaling as well as the stiffness scaling provide a quasi optimal condition number bound according to \thmref{thm:CondCoeffScal} and \thmref{thm:CondStiffScal}.
\begin{table}[h!]
\centering
\begin{tabular}{|c|c|cc|cc|cc|}\hline
 \mc{2}{|c|}{ALG. A} & \mc{2}{|c|}{unprec. F} & \mc{2}{|c|}{coeff. scal} & \mc{2}{|c|}{stiffness scal.} \\ \hline
$\#$dofs & $H/h$ & $\kappa$ & It. & $\kappa$ & It. & $\kappa$ & It. \\ \hline
2364 & 9 & 45 & 50 & 9 &  {21} & 9 &  {20}\\
4728 & 13 & 73 & 46 & 11 &  {22} & 11 &  {22}\\
11856 & 21 & 133 & 57 & 15 &  {24} & 14 &  {24}\\
35712 & 37 & 265 & 68 & 18 &  {25} & 18 &  {25} \\ \hline\hline
\mc{2}{|c|}{ALG. C} & \mc{6}{|c|}{ }\\ \hline
 2364& 9& 28 & 44 & 1.8 &  {11} & 1.8 &  {11}\\
 4728& 13& 22 & 39  & 2 &  {12} & 2 &  {12}\\
 11856& 21 & 18 & 39 & 2.4 &  {14} & 2.4 &  {14}\\
 35712& 37& 17 & 38 & 2.8 &  {15} & 2.8 &  {15} \\ \hline
  \end{tabular}
  \caption{2D example with $\deg = 4$ and homogeneous diffusion coefficient. Dependence of the condition number $\kappa$ and the number It. of iterations on $H/h$ for the unpreconditioned system and preconditioned system with coefficient and stiffness scaling. 
  Choice of primal variables: vertex evaluation (upper table), vertex evaluation and edge averages (lower table).}
\label{table:hom2D}
  \end{table}
\begin{table}[h!]
\centering
\begin{tabular}{|c|c|cc|cc|cc|}\hline
 \mc{2}{|c|}{ALG. A} & \mc{2}{|c|}{unprec. F} & \mc{2}{|c|}{coeff. scal} & \mc{2}{|c|}{stiffness scal.} \\ \hline
		$\#$dofs & $H/h$ & $\kappa$ & It. & $\kappa$ & It. & $\kappa$ & It. \\ \hline
		7548   &  5 & 3254  & 393 & 63 &  {33} & 63 &  {33}\\
		14368  & 7  & 3059  & 356 & 86  &  {37} & 86  &  {37}\\
		38100  & 10 & 2170  & 317 & 196 &  {45} & 196 &  {46}\\
		142732 & 16 & 7218  & 397 & 467 &  {64} & 468 &  {65}\\ \hline\hline 

		      \mc{2}{|c|}{ALG. B} & \mc{6}{|c|}{ }\\ \hline
		7548 &  5 &  2751 & 341 & 1.6 &  {10} & 1.6 &  {10}\\
		14368 &  7 & 2860 & 397 & 1.7 &  {11} & 1.7 &  {11}\\
		38100 & 10 & 1697 & 333 & 2.0 &  {12} & 2.3 &  {13}\\
		142732& 16 & 1261 & 333 & 2.3 &  {13} & 3.1 &  {16}\\\hline
  \end{tabular}
  \caption{3D example with $\deg = 4$ homogeneous diffusion coefficient. Dependence of the condition number $\kappa$ and the number It. of iterations on $H/h$ for the unpreconditioned system and preconditioned system with coefficient and stiffness scaling. Choice of primal variables: vertex evaluation (upper table), vertex evaluation and edge averages and face averages (lower table).}
\label{table:hom3D}
  \end{table}
\subsection{The case of jumping diffusion coefficient}
  We investigate numerical numerical examples with patchwise constant diffusion coefficient $\alpha$, with a jumping pattern according to \figref{fig:coeffjumps}. The values of the diffusion coefficient are $10^{-3}$ (blue) and $10^3$ (red).  The 2D results are summarized in \tabref{table:jump2D} and the 3D results in  \tabref{table:jump3D}. 
We observe a quasi optimal condition number bound which is clearly independent of the diffusion coefficient and its jumps across subdomain interfaces. 
  \begin{figure}[h!]
  \begin{center}
   \includegraphics[width=0.6\textwidth]{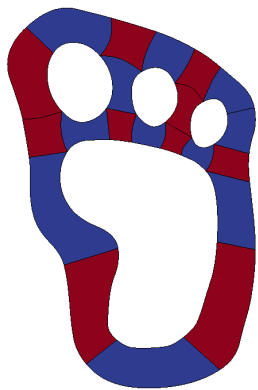}
  \caption{Pattern of the jumping diffusion coefficient}
    \label{fig:coeffjumps}
    \end{center}
  \end{figure}
\begin{table}[h!]
\centering
\begin{tabular}{|c|c|cc|cc|cc|}\hline

      \mc{2}{|c|}{ALG. A} & \mc{2}{|c|}{unprec. F} & \mc{2}{|c|}{coeff. scal} & \mc{2}{|c|}{stiffness scal.} \\ \hline
	      $\#$dofs & $H/h$ & $\kappa$ & It. & $\kappa$ & It. & $\kappa$ & It. \\ \hline
		  2364   & 9  & 1.4e07  & 317 & 5.6   &  {13} & 5.3   &  {13}\\
		  4728   & 13 & 1.5e07  & 297 & 7.0   &  {13} & 6.4   &  {13}\\
		  11856  & 21 & 2.4e07  & 397 & 8.7   &  {15} & 7.8  &  {13}\\
		  35712  & 37 & 4.0e07  & 434 & 10.6  &  {16} & 9.3  &  {14} \\ \hline\hline
		  
		  \mc{2}{|c|}{ALG. C} & \mc{6}{|c|}{ }\\ \hline
		  2364   & 9  & 1.5e07   & 261 & 1.8 &  {7} & 1.7 &  {7}\\
		  4728   & 13 & 1.1e07   & 267 & 2.2 &  {8} & 2   &  {7}\\
		  11856  & 21 & 9.8e06   & 291 & 2.6 &  {8} & 2.3 &  {8}\\
		  35712  & 37 & 9.0e06   & 310 & 3.0 &  {10}& 2.7 &  {10} \\ \hline
  \end{tabular}
    \caption{2D example with $\deg = 4$ jumping diffusion coefficient. Dependence of the condition number $\kappa$ and the number It. of iterations on $H/h$ for the unpreconditioned system and preconditioned system with coefficient and stiffness scaling. Choice of primal variables: vertex evaluation (upper table), vertex evaluation and edge averages (lower table).}
\label{table:jump2D}
  \end{table}
\begin{table}[h!]
\centering
\begin{tabular}{|c|c|cc|cc|cc|}\hline
      \mc{2}{|c|}{ALG. A} & \mc{2}{|c|}{unprec. F} & \mc{2}{|c|}{coeff. scal} & \mc{2}{|c|}{stiffness scal.} \\ \hline
	      $\#$dofs & $H/h$ & $\kappa$ & It. & $\kappa$ & It. & $\kappa$ & It. \\ \hline
	      	7548   & 5  & >1.e16  & >1000 & 47  &  {20} & 47  &  {18}\\
		14368  & 7  & >1.e16  & >1000 & 69  &  {20} & 65  &  {19}\\
		38100  & 10 & >1.e16  & >1000 & 165 &  {32} & 152  &  {29}\\
		142732 & 16 & >1.e16  & >1000 & 405 &  {38} & 368 &  {34}\\ \hline\hline 

		      \mc{2}{|c|}{ALG. B} & \mc{6}{|c|}{ }\\ \hline
		7548  & 5  & >1.e16 & >1000 & 1.7  &  {7} & 1.6  &  {7}\\
		14368 &  7 & >1.e16 & >1000 & 1.8 &  {7} & 1.7 &  {7}\\
		38100 & 10 & >1.e16 & >1000 & 2.1 &  {8} & 2.3 &  {8}\\
		142732& 16 & >1.e16 & >1000 & 4.4 &  {9} & 3.2 &  {11}\\\hline
		
  \end{tabular}
   \caption{3D example with $\deg = 4$ jumping diffusion coefficient. Dependence of the condition number $\kappa$ and the number It. of iterations on $H/h$ for the unpreconditioned system and preconditioned system with coefficient and stiffness scaling. Choice of primal variables: vertex evaluation (upper table), vertex evaluation and edge averages and face averages (lower table).}
\label{table:jump3D}
  \end{table}
\subsection{Dependence on $\deg$}
We want to examine the dependence of the condition number on the B-Spline degree $\deg$, although the presented theory in \secref{sec:AnalysisBDDC} does not cover the dependence of IETI-DP preconditioned system on $\deg$. 
We note that in our implementation the degree elevation yields an increase in the multiplicity of the knots within each step, 
resulting in $C^1$ smoothness on each patch.
The computational domain is chosen as the 2D YETI-footprint presented in \figref{fig:YETI_footprint} and the diffusion coefficient is chosen to be globally constant.  The results are summarized in \tabref{table:pDependence2D}, where we observe a possibly logarithmic dependence of the condition number on the polynomial degree in case of the coefficient scaling as well as of the above mention version of the stiffness scaling. The numerical experiments depict a linear dependence in case of the regular stiffness scaling \equref{sec:IETI_precond}, see \figref{fig:pDependence2D}. 
\begin{table}[h!]
\centering
\begin{tabular}{|c| c|cc|cc|cc|cc|}\hline
      \mc{2}{|c|}{ALG. C} & \mc{2}{|c|}{unprec. F} & \mc{2}{|c|}{coeff. scal} & \mc{2}{|c|}{stiffness scal.}& \mc{2}{|c|}{stiff. scal. modif.} \\ \hline
	      $\#$dofs & degree & $\kappa$ & It. & $\kappa$ & It. & $\kappa$ & It. & $\kappa$ & It.\\ \hline
800    &  2  &  3.24  & 16    &   1.65  & 8    &     1.64  & 8  & 1.63 & 8   \\      
2364   &  3  &  8.08  & 28    &   1.71  & 10   &     1.69  & 10 & 1.7  & 10  \\
4728   &  4  &  24.2  & 51    &   1.83  & 11   &     1.88  & 12 & 1.83 & 11  \\
7892   &  5  &  82.8  & 86    &   2.03  & 13   &     2.16  & 12 & 2.01 & 12  \\ 
11856  &  6  &  296   & 140   &   2.23  & 13   &     2.42  & 14 & 2.18 & 13  \\
16620  &  7  &  1082   & 230   &   2.41  & 14   &     2.66  & 14 & 2.34 & 14  \\ 
22184  &  8  &  4021   & 371   &   2.57  & 15   &     2.88  & 15 & 2.49 & 15  \\ 
28548  &  9  &  15034  & 594   &   2.72  & 15   &     3.09  & 16 & 2.63 & 15  \\
35712  &  10 &  56773  & 968   &   2.87  & 16   &     3.28  & 16 & 2.75 & 15  \\ \hline 	
  \end{tabular}
   \caption{2D example with fixed initial mesh and homogeneous diffusion coefficient. Dependence of the condition number $\kappa$ and the number It. of iterations on $H/h$ for the unpreconditioned system and preconditioned system with coefficient and stiffness scaling. Choice of primal variables: vertex evaluation and edge averages.}
\label{table:pDependence2D}
  \end{table}
    \begin{figure}[h!]
  \begin{center}
   \includegraphics[width=0.8\textwidth]{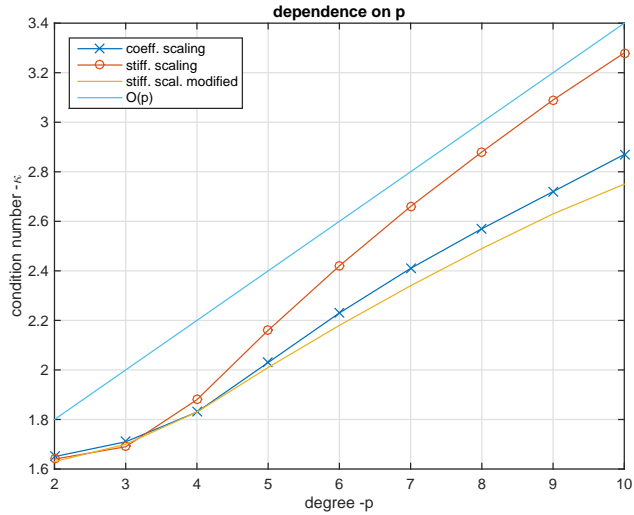}
  \caption{2D example with fixed initial mesh and homogeneous diffusion coefficient. Condition number $\kappa$ as a function of polynomial degree. Choice of primal variables: vertex evaluation and edge averages.}
    \label{fig:pDependence2D}
    \end{center}
  \end{figure}
  \subsection{Performance}
  The algorithm was tested on a Desktop PC with an Intel(R) Xeon(R) CPU E5-1650 v2 @ 3.50GHz and 16 GB main memory. As already mentioned at the beginning of this section, we used the open source library G+SMO for the materialization of the code. Moreover, we make use of the Sparse-LU factorization of the open source library ``Eigen''\footnote{\url{http://eigen.tuxfamily.org/index.php?title=Main_Page}} for the local solvers. The timings presented in \tabref{table:timings} are obtained from a sequential implementation of the code. We choose the same setting as presented in \tabref{table:jump2D} with Algorithm C. However, we do one more refinement steps and obtain $121824$ total degrees of freedom, $1692$ Lagrange multipliers, and on each patch approximate $4900$ local degrees of freedom. We select a run with coefficient scaling and obtain a condition number of $\kappa = 3.53$ and $11$ iterations.
  
  We remark that about $90\,\%$ of the runtime is used for the assembling part of the program including the Schur complement computations. In more detail, most of the time is spent for calculating the LU-factorizations of the local matrices. This indicates the importance of replacing the direct solver with inexact solvers on each patch, see, e.g., \cite{widlundLi2007_inexactSolvers},\cite{KlawonnRheinbach2007_inexactSolvers}. Furthermore, we want to note, that especially in 3D, an additional bottleneck is the memory demand of the direct solvers. 
  \begin{table}[h!]
   \centering
   \begin{tabular}{|l|r|r|}
   \hline
      & Wall-clock time & relative time in \% \\ \hline \hline
    Preparing the bookkeeping & 0.011 s &  0.03 \\ \hline\hline
    Assembling all patch local $\Kb\sMP$ &  6.2 s & 15.42 \\
    Partitioning w.r.t. $B$ and $I$ & 0.087 s & 0.22\\
    Assembling C & 0.016 s &0.04 \\
    Calculating LU -fact. of $\KII$& 15 s & 37.31\\
    Calculating LU -fact. of $\MatTwo{K\sMP}{{C\sMP}^T}{C\sMP}{0}$& 15 s & 37.31\\
    Assembling and LU-fact of $\SPP$ & 0.46 s& 1.14\\
    Assemble rhs. &  0.094 s& 0.23 \\ \hline
    Total assembling & 37 s & 92.04  \\ \hline \hline
    One PCG iteration & 0.22 s & - \\ \hline
    Solving the system & 2.5 s & 6.22 \\ \hline \hline
    Calculating the solution $\ub$ & 0.5 s& 1.24 \\ \hline\hline
    Total spent time& 40.2 s & 100.00 \\\hline
   \end{tabular}
   \caption{Timings of the 2D example  with coefficient scaling, Algorithm C and one more refinement compared to \tabref{table:jump2D}. The discrete problem consists of $121824$ total degrees of freedom, $1692$ Lagrange multipliers, and on each patch approximate $4900$ local degrees of freedom. Middle column presents the absolute spent time and right column the relative one.}
\label{table:timings}
  \end{table}
\section{Conclusions}
\label{sec:conclusions}
We have derived condition number estimates for the IETI-DP method and extended the existent theory to domains which cannot be represented by a single geometrical mapping. Due to the fact, that we only considered open knot vectors, we could identify basis function on the interface and on the interior. This assumption implies that the discrete solution is only $C^0$ smooth across patch interfaces. However, under this assumption, we were able to find an improved condition number bound of the IETI-DP method using the Dirichlet preconditioner with stiffness scaling.
Numerical examples with two and three dimensional domains, different choices of primal variables and different scaling methods confirmed the theoretical results presented in \secref{sec:IETI}. Additionally, we investigated the B-Spline degree dependence of the preconditioned system. 
Moreover, the numerical results indicate the robustness with respect to jumping
diffusion coefficients across the interfaces.
We have obtained similar numerical results for solving multipatch discontinuous Galerkin (dG) IgA schemes, proposed and investigated in \cite{HL:LangerToulopoulos:2014a}, by means of IETI-DP methods following the approach developed by  \cite{dryja2013fetiDP} for a composite finite element and dG method. Our results in the IETI-DP versions for solving multipatch dG IgA equations will be published in a forthcoming paper.
%
%
%
%

\bibliographystyle{unsrt}

\bibliography{IETI_bib.bib}

\end{document}

%% file: Makros.tex
  \usepackage{xifthen}

  \newcommand{\genref}[2]{#2\,\ref{#1}}
 
  \newcommand{\thmref}[1]{\genref{#1}{Theorem}}
  \newcommand{\lemref}[1]{\genref{#1}{Lemma}}

  \newcommand{\propref}[1]{\genref{#1}{Proposition}}
  \newcommand{\corref}[1]{\genref{#1}{Corrolary}}
  \newcommand{\equref}[1]{(\ref{#1})} 
  \newcommand{\secref}[1]{Section\,\ref{#1}}
  \newcommand{\tabref}[1]{Table\,\ref{#1}}
  \newcommand{\figref}[1]{Figure\,\ref{#1}}
 
 \newcommand{\mc}[3]{\multicolumn{#1}{#2}{#3}}
 
  \newcommand{\MatThree}[9]
  {
    \begin{bmatrix}
    #1 & #2 & #3\\
    #4 & #5 & #6\\
    #7 & #8 & #9
    \end{bmatrix}
  }
  
  \newcommand{\MatTwo}[4]
  {\begin{bmatrix}
    #1 & #2\\
    #3 & #4
  \end{bmatrix}}
  
  \newcommand{\VecTwo}[2]
  {\begin{bmatrix}
    #1 \\
    #2 
  \end{bmatrix}}
  \newcommand{\VecThree}[3]
  {\begin{bmatrix}
    #1 \\
    #2 \\
    #3
  \end{bmatrix}}

  \newcommand{\DN}[1]{\frac{\partial #1}{\partial n}}
  \newcommand{\dd}[2]{\frac{\partial #1}{\partial #2}}
  \newcommand{\real}[1][]
  {
    \ifthenelse{\isempty{#1}}%
    {\mathbb{R}}
    {\mathbb{R}^{{#1}}}
  }
    \newcommand{\nat}[1][]
  {
    \ifthenelse{\isempty{#1}}%
    {\mathbb{N}}
    {\mathbb{N}^{#1}}
  }
  \DeclareMathOperator{\mdiv}{div}
  \DeclareMathOperator{\grad}{\nabla}
  \renewcommand{\ker}[1]{\text{ker}(#1)}
    \newcommand{\supp}{\text{supp}}
    \newcommand{\spn}{\text{span}}
    
    \newcommand{\diam}{\text{diam}}
    \newcommand{\for}{\mbox{for }}
    \newcommand{\range}[1]{\mathcal{R}(#1)}
  
  \newcommand{\bil}[3][]
  {
  \ifthenelse{\isempty{#1}}%
    {a(#2,#3)}
    {a\sMP[#1](#2,#3)}
  }
  \newcommand{\func}[3][]
  {
  \ifthenelse{\isempty{#1}}%
    {\left\langle #2, #3 \right\rangle}
    {\left\langle #2\sMP[#1], #3 \right\rangle}
  }
  
  \newcommand{\mat}[1]{\boldsymbol{#1}}
  \renewcommand{\vec}[1]{\boldsymbol{#1}}
  
  \newcommand{\LTwoNorm}[2]{\left\|#1\right\|_{L^2(#2)}}
  \newcommand{\LInfNorm}[2]{\left\|#1\right\|_{L^\infty(#2)}}

  \newcommand{\HSNorm}[3]{|#2|_{H^{#1}(#3)}}
  \newcommand{\HOneSNorm}[2]{\HSNorm{1}{#1}{#2}}

  \newcommand{\DiscLLocNorm}[2]{|#1|_{#2}}
  \newcommand{\DiscLNorm}[1]{|#1|_{\Square}}
  \newcommand{\DiscHNorm}[1]{|#1|_{\nabla}}
  
  \newcommand{\RealSNorm}[1]{\left\|\!\left|#1\right\|\!\right|_{\nabla}}
  \newcommand{\RealTNorm}[1]{\left\|\!\left|#1\right\|\!\right|_{e}}
    
\newcommand{\x}[1][]{\ifthenelse{\isempty{#1}}{x}{x_{#1}}}

\newcommand{\g}[1]{\gBSplFormat{#1}}

\newcommand{\pDim}{{d}}
\newcommand{\gDim}{g}
\newcommand{\pDom}[1][]{\ifthenelse{\isempty{#1}}{(0,1)^\pDim}{\ifthenelse{\equal{#1}{2}}{[0,1]^2}{(0,1)^\pDim}}}
\newcommand{\gDom}{\Omega}

\newcommand{\intBaseLetter}{\Gamma}
\newcommand{\gIntG}[2]{\intBaseLetter^{(#1,#2)}}
\newcommand{\gInt}{\gIntG{k}{l}}
\newcommand{\gIntMP}{\intBaseLetter}

\newcommand{\nMP}{N}
\newcommand{\iMP}{{k}}
\newcommand{\sMP}[1][]{\ifthenelse{\isempty{#1}}{^{(\iMP)}}{^{(#1)}}}

\newcommand{\iSet}{{\mathcal{I}}}
\newcommand{\iSetMesh}{\iSet_{\pMeshBase}}

\newcommand{\iVar}{\iota}

\newcommand{\kn}[1][]{\ifthenelse{\isempty{#1}}{\xi}{\xi_{#1}}}
\newcommand{\knV}{\Xi}
\newcommand{\nkn}{m}

\newcommand{\nIntKnotspan}{\nu}
\renewcommand{\deg}{p}
\newcommand{\mul}{r}
\newcommand{\IndexSet}{\mathcal{B}}

\newcommand{\BSplBaseLetter}{N}
\newcommand{\gBSplFormat}[1]{\check{#1}}
\newcommand{\gBSplBaseLetter}{\gBSplFormat{\BSplBaseLetter}}

\newcommand{\BSplG}[2]{\BSplBaseLetter_{#1,#2}}
\newcommand{\BSpl}{\BSplG{i}{\deg}}
\newcommand{\BSplFamG}[2]{\{\BSplG{#1}{#2}\}_{#1\in\iSet}}
\newcommand{\BSplFam}{\BSplFamG{i}{\deg}}

\newcommand{\gBSplG}[2]{\gBSplBaseLetter_{#1,#2}}
\newcommand{\gBSpl}{\gBSplG{i}{\deg}}
\newcommand{\gBSplFamG}[3][]{\ifthenelse{\isempty{#1}}
{\{\gBSplG{#2}{#3}\}_{#2\in\iSet}}
{\{\gBSplG{#2}{#3}\sMP[#1]\}_{#2\in\iSet\sMP[#1]}}}

\newcommand{\gBSplFam}{\gBSplFamG{i}{\deg}}
\newcommand{\nBasis}{M}

\newcommand{\pBaseFormat}[1]{\hat{#1}}
\newcommand{\meshBaseLetter}{Q}
\newcommand{\meshBaseFormat}{\mathcal{\meshBaseLetter}}
\newcommand{\pMeshBase}{\pBaseFormat{\meshBaseFormat}}
\newcommand{\pMesh}{\pMeshBase_h}
\newcommand{\gMesh}{\meshBaseFormat_h}
\newcommand{\gCell}[1][]{\ifthenelse{\isempty{#1}}{\meshBaseLetter}{\meshBaseLetter_{#1}}}
\newcommand{\pCell}[1][]{\ifthenelse{\isempty{#1}}{\pBaseFormat{\meshBaseLetter}}{\pBaseFormat{\meshBaseLetter}_{#1}}}

\newcommand{\Cp}[1][]{\ifthenelse{\isempty{#1}}{P_{i}}{P_{#1}}}
\newcommand{\geoMap}{G}

\newcommand{\baseV}{V}

\newcommand{\genSp}[3][]{\ifthenelse{\isempty{#1}}{#2_{#3}}{#2_{#3,#1}}}
\newcommand{\gVh}{\genSp{\baseV}{h}}
\newcommand{\pVh}{\genSp{\mathrm S}{h}}
\newcommand{\gVB}{\genSp[h]{\baseV}{\Gamma}}
\newcommand{\gVI}{\genSp[h]{\baseV}{I}}

\newcommand{\VD}[1][]{\genSp[#1]{\baseV}{D}}

\newcommand{\VDh}{\VD[h]}

\newcommand{\genLocSpace}[3][]{\ifthenelse{\isempty{#1}}{#2_{#3}}{#2_{#3}\sMP[#1]}}
\newcommand{\IntermSpaceForamt}[1]{\widetilde{#1}}
\newcommand{\ContSpaceFormat}[1]{\widehat{#1}}

\newcommand{\base}{W}
\newcommand{\hW}{\ContSpaceFormat{\base}}
\newcommand{\tW}{\IntermSpaceForamt{\base}}

\newcommand{\sB}{_B}
\newcommand{\sI}{_I}

\newcommand{\setB}{\iSet\sB} 
  
  \newcommand{\emb}{\IntermSpaceForamt{I}}
  \newcommand{\tB}{\IntermSpaceForamt{B}}
  \newcommand{\tg}{\IntermSpaceForamt{g}}
  \newcommand{\tphi}{\IntermSpaceForamt{\phi}}
  \newcommand{\nLag}{\Lambda}
  \newcommand{\LamSet}{U}
  \newcommand{\tLamSet}{\IntermSpaceForamt{U}}

 \newcommand{\LocVB}[1][]{\ifthenelse{\isempty{#1}}{\base}{\ifthenelse{\equal{#1}{L}}{\overline{\base}}{\base\sMP[#1]}}}
  \newcommand{\Wp}[1][]{\genLocSpace[#1]{\base}{\Pi}}     
  \newcommand{\Wd}[1][]{\genLocSpace[#1]{\base}{\Delta}}  
  \newcommand{\tWp}[1][]{\genLocSpace[#1]{\tW}{\Pi}}
  \newcommand{\tWd}[1][]{\genLocSpace[#1]{\tW}{\Delta}}

 \newcommand{\GenOp}[2][]{
  \ifthenelse{\isempty{#1}}
    {#2}
    {#2\sMP[#1]}
  }

  \newcommand{\Sop}[1][]{\GenOp[#1]{S}}
  \newcommand{\tSop}[1][]{\IntermSpaceForamt{\Sop{#1}}} 
  
  \newcommand{\Kop}[1][]{\GenOp[#1]{K}}
  \newcommand{\Ci}[1][]{\GenOp[#1]{C}}
  
  \newcommand{\np}[1][]{ \ifthenelse{\isempty{#1}} {n_{\Pi}} {n_{\Pi}\sMP[#1]} }
  
  \newcommand{\GenMatXX}[3][]
  {
  \ifthenelse{\isempty{#1}}%
    {#2_{#3}}
    {#2_{#3}\sMP[#1]}
  }

  \newcommand{\KBB}[1][]{ \GenMatXX[#1]{\Kb}{BB}}
  \newcommand{\KIB}[1][]{ \GenMatXX[#1]{\Kb}{IB}}
  \newcommand{\KBI}[1][]{ \GenMatXX[#1]{\Kb}{BI}}  
  \newcommand{\KII}[1][]{ \GenMatXX[#1]{\Kb}{II}}
  
  \newcommand{\KBBop}[1][]{ \GenMatXX[#1]{\Kop}{BB}}
  \newcommand{\KIBop}[1][]{ \GenMatXX[#1]{\Kop}{IB}}
  \newcommand{\KBIop}[1][]{ \GenMatXX[#1]{\Kop}{BI}}  
  \newcommand{\KIIop}[1][]{ \GenMatXX[#1]{\Kop}{II}}
  
  \newcommand{\SPP}[1][]{ \GenMatXX[#1]{\Sb}{\Pi\Pi}}

  \newcommand{\SPPop}[1][]{ \GenMatXX[#1]{\Sop}{\Pi\Pi}}
  \newcommand{\SDPop}[1][]{ \GenMatXX[#1]{\Sop}{\Delta\Pi}}
  \newcommand{\SPDop}[1][]{ \GenMatXX[#1]{\Sop}{\Pi\Delta}}
  \newcommand{\SDDop}[1][]{ \GenMatXX[#1]{\Sop}{\Delta\Delta}}

  \newcommand{\RestrBD}[1][]{ \GenMatXX[#1]{{R}}{B\Delta}}
  \newcommand{\RestrBP}[1][]{ \GenMatXX[#1]{{R}}{B\Pi}}
  \newcommand{\RestrD}[1][]{ \GenMatXX[#1]{{R}}{\Delta}}  
  \newcommand{\RestrP}[1][]{ \GenMatXX[#1]{{R}}{\Pi}}
  
  \newcommand{\RestrScB}[1][]{ \GenMatXX[#1]{{R}}{D,B}}
  \newcommand{\RestrScD}[1][]{ \GenMatXX[#1]{{R}}{D,\Delta}}
     
  \newcommand{\wP}[1][]
  {
    \ifthenelse{\isempty{#1}}%
    {\boldsymbol{w}_{\Pi}}
    {\boldsymbol{w}_{\Pi}\sMP[#1]}
  }

  \newcommand{\wD}[1][]
  {
    \ifthenelse{\isempty{#1}}%
    {w_{\Delta}}
    {w_{\Delta}\sMP[#1]}
  }
  \newcommand{\fP}[1][]
  {
    \ifthenelse{\isempty{#1}}%
    {\boldsymbol{f}_{\Pi}}
    {\boldsymbol{f}_{\Pi}\sMP[#1]}
  }
  \newcommand{\fD}[1][]
  {
    \ifthenelse{\isempty{#1}}%
    {f_{\Delta}}
    {f_{\Delta}\sMP[#1]}
  }

  \newcommand{\Scal}[1][]{  \ifthenelse{\isempty{#1}}{{\delta^{\dagger}}}{{\delta_{#1}^{\dagger}}}}
  
  \newcommand{\DHE}[1][]{  \ifthenelse{\isempty{#1}}{\mathcal{H}}{\mathcal{H}\left(#1\right)}}
  \newcommand{\DHEreal}[1][]{  \ifthenelse{\isempty{#1}}{\mathcal{H}_{\mathcal{Q}_1}}{\mathcal{H}_{\mathcal{Q}_1}\left(#1\right)}}
  
 \newcommand{\tmu}{\IntermSpaceForamt{\boldsymbol{\mu}}}

 \newcommand{\Fb}{\boldsymbol{F}} 
 
 \newcommand{\db}{\boldsymbol{d}}
 \newcommand{\gb}{\boldsymbol{g}}
 \newcommand{\lambdavec}{\boldsymbol{\lambda}}

 \newcommand{\evec}{\boldsymbol{e}}

 \newcommand{\GenBold}[2][]{
  \ifthenelse{\isempty{#1}}
    {\boldsymbol{#2}}
    {\boldsymbol{#2}\sMP[#1]}
  }
  \newcommand{\GenBoldUnderscore}[3][]{
  \ifthenelse{\isempty{#1}}
    {\boldsymbol{#2}#3}
    {\boldsymbol{#2}#3\sMP[#1]}
  }
  
 \newcommand{\Ib}{\GenBold{I}}
 
 \newcommand{\ub}[1][] {\GenBold[#1]{u}}
 \newcommand{\wb}[1][] {\GenBold[#1]{w}}
 \newcommand{\vb}[1][] {\GenBold[#1]{v}}

 \newcommand{\gvec}[1][]  {\GenBold[#1]{g}}
 \newcommand{\fvec}[1][]  {\GenBold[#1]{f}}
 \newcommand{\Kb}[1][]    {\GenBold[#1]{K}}
 \newcommand{\Sb}[1][]    {\GenBold[#1]{S}}

  \newcommand{\fvecB}[1][] {\GenBoldUnderscore[#1]{f}{\sB} }
  \newcommand{\fvecI}[1][] {\GenBoldUnderscore[#1]{f}{\sI} }
  
  \newcommand{\uIb}[1][] {\GenBoldUnderscore[#1]{u}{\sI} }
  \newcommand{\uBb}[1][] {\GenBoldUnderscore[#1]{u}{\sB} }